\documentclass[11pt]{amsart}
\usepackage[leqno]{amsmath}
\usepackage{amssymb}
\usepackage{enumerate}
\usepackage{subfigure}
\usepackage{graphicx}
\usepackage[all]{xy}
\usepackage[usenames,dvipsnames]{xcolor}
\usepackage[
colorlinks=true,linkcolor=NavyBlue,urlcolor=RoyalBlue,citecolor=PineGreen,bookmarks=true,bookmarksopen=true,bookmarksopenlevel=2,unicode=true,linktocpage]{hyperref}

\usepackage{microtype}
\usepackage{centernot}
\usepackage{stmaryrd}
\usepackage{comment}
\usepackage{bm}

\numberwithin{equation}{section}

\newtheorem{theorem}{Theorem}[section]
\newtheorem{remark}[theorem]{Remark}
\newtheorem{lemma}[theorem]{Lemma}
\newtheorem{proposition}[theorem]{Proposition}

\newtheorem{definition}[theorem]{Definition}

\newcommand{\E}{\mathbf{E}}

\newcommand{\N}{\mathbf{N}}
\newcommand{\Z}{\mathbf{Z}}
\newcommand{\p}{\mathbf{P}}

\newcommand{\R}{\mathbf{R}}

\newcommand{\SLE}{{\rm SLE}}
\newcommand{\CLE}{{\rm CLE}}

\newcommand{\var}{\mathrm{var}}

\newcommand{\one}{{\bf 1}}

\newcommand{\wt}{\widetilde}
\newcommand{\wh}{\widehat}
\newcommand{\ol}{\overline}

\newcommand{\NW}{\mathsf{NW}}
\newcommand{\NE}{\mathsf{NE}}
\newcommand{\SE}{\mathsf{SE}}
\newcommand{\SW}{\mathsf{SW}}

\newcommand{\RPM}{\mathsf{RPM}}
\newcommand{\RPMs}{\mathsf{RPMs}}
\newcommand{\BPRPM}{\mathsf{BPRPM}}
\newcommand{\BPRPMs}{\mathsf{BPRPMs}}

\newcommand{\move}{{\mathfrak m}}
\newcommand{\weight}{{\mathfrak a}}
\newcommand{\id}{\mathrm{id}}

\begin{document}

\title[Bipolar oriented random planar maps with large faces and exotic $\SLE_\kappa(\rho)$]{Bipolar oriented random planar maps with large faces and exotic $\SLE_\kappa(\rho)$ processes}

\author{Konstantinos Kavvadias and Jason Miller}

\begin{abstract}
We consider bipolar oriented random planar maps with heavy-tailed face degrees.  We show for each $\alpha \in (1,2)$ that if the face degree is in the domain of attraction of an $\alpha$-stable L\'evy process, the corresponding random planar map has an infinite volume limit in the Benjamini-Schramm topology.  We also show in the limit that the properly rescaled contour functions associated with the northwest and southeast trees converge in law to a certain correlated pair of $\alpha$-stable L\'evy processes.  Combined with other work, this allows us to identify the scaling limit of the planar map with an $\SLE_\kappa(\rho)$ process with $\rho = \kappa-4 < -2$ on $\sqrt{\kappa}$-Liouville quantum gravity for $\kappa \in (4/3,2)$ where $\alpha, \kappa$ are related by $\alpha = 4/\kappa-1$.
\end{abstract}

\date{\today}
\maketitle

\setcounter{tocdepth}{1}
\tableofcontents

\parindent 0 pt
\setlength{\parskip}{0.20cm plus1mm minus1mm}

\section{Introduction}
\label{sec:introduction}

\subsection{Overview and setting}

A \emph{planar map} is a graph $G = (V,E)$ together with an embedding into the plane so that no two edges cross, considered up to orientation preserving homeomorphism.  A \emph{random planar map} ($\RPM$) is a planar map chosen according to some probability measure.  Well studied examples of $\RPMs$ include planar triangulations or quadrangulations chosen uniformly at random or with the extra structure of a statistical mechanics model, such as the uniform spanning tree (see \cite{lg2019survey,miermontstflour,dms2014mating,s2016qginv} and the references therein).  In this work, we will study $\RPMs$ with the extra structure of a \emph{bipolar orientation} with a specified source and sink (the so-called ``poles'').  Recall that a bipolar orientation of a graph $G$ is an acyclic orientation of its edges without a source or sink except at the specified poles.  The source (resp.\ sink) is a vertex with no incoming (resp.\ outgoing) edges.

Bipolar oriented planar maps were previously studied in \cite{kmsw2019bipolar}, in which it was shown that they are in bijection with a walk on $\Z^2$ whose coordinates are the contour functions for a pair of discrete trees.  These discrete trees are concretely defined as follows.  Suppose that $G$ is a planar map with  no self-loops but with multiple edges allowed which is equipped with a bipolar orientation.  Then we can view the directed edges as a pointing ``north'' while their opposites point ``south''.  As explained in \cite[Section~1.2]{kmsw2019bipolar}, the edges pointing to an interior vertex given in cyclic order consists of a single group of north-going edges and a single group of south-going edges.  This implies that each vertex has a unique north-going edge which is west-most, which is called its northwest ($\NW$) edge, and this allows one to define the $\NW$ tree in which one takes the parent of an edge $e$ to be the $\NW$ edge from the upper vertex of $e$.  The southeast ($\SE$) tree is defined analogously except with the $\SE$ edge in place of the $\NW$ edge.  The root of the $\NW$ (resp.\ $\SE$) tree is the source (resp.\ sink).

By using the contour function encoding, the $\NW$ and $\SE$ trees correspond to a walk on $\Z^2$.  It was shown in \cite{kmsw2019bipolar} that if one picks a bipolar oriented planar map uniformly at random ($\BPRPM$) with a fixed number of edges, all of which are triangles then the contour functions for the $\NW$ and $\SE$ trees converge in the limit to a correlated two-dimensional Brownian motion.  Using the mating of trees machinery developed in \cite{dms2014mating}, this result has the interpretation of being a scaling limit result for such maps towards the Schramm-Loewner evolution ($\SLE$) \cite{schramm2000scaling} with parameter $\kappa=12$ on a $\sqrt{4/3}$-Liouville quantum gravity (LQG) surface.  The results of \cite{kmsw2019bipolar} hold more generally if the bipolar oriented map consists of faces which all have the same size or are allowed to have varying sizes, as long as the face degree distribution has a sufficiently strong tail.  (See also the works \cite{ghs2016bipolar,bm2020baxter}, which prove the joint convergence of the $\NW$, $\NE$, $\SW$, $\SE$ trees where the $\SW$ and $\NE$ trees are defined using the so-called dual orientation.)

In this work, we will consider $\BPRPMs$ with \emph{large faces}, meaning that the law of the face degree is in the domain of attraction of an $\alpha$-stable random variable with $\alpha \in (1,2)$.  In this regime, we will construct the infinite volume limit and then show that the properly rescaled encoding contour functions converge in the limit to a certain pair of correlated $\alpha$-stable L\'evy processes.  When combined with another work, this will allow us to interpret certain exotic $\SLE_\kappa(\rho)$ processes \cite{ms2019lightcone} (with $\rho=\kappa-4 < -2$) on $\sqrt{\kappa}$-LQG as the scaling limit of $\BPRPMs$ with large faces where the relationship between $\alpha,\kappa$ is given by $\alpha = 4/\kappa-1$ and $\kappa \in (4/3,2)$.  We recall that the standard $\SLE_\kappa(\rho)$ processes (i.e., $\rho > -2$) and $\kappa \in (0,4]$ are simple curves and have dimension $1+\kappa/8$ \cite{bef2008dimension,rs2005basic} but the $\SLE_\kappa(\rho)$ processes with $\kappa \in (0,4]$ and $\rho < -2$ are self-intersecting \cite{ms2019lightcone,msw2017clepercolations} and have dimension strictly larger than $1+\kappa/8$ \cite{m2018dimension}.  In the particular case $\kappa \in (4/3,2)$ so that $\rho = \kappa-4 \in (-\kappa/2-2,-2)$, the dimension is given by $1+2/\kappa - 3\kappa/8 > 1+ \kappa/8$ \cite{m2018dimension}.

We remark that random planar maps with large faces have been studied in several other contexts.  For example, the work \cite{lgm2011largefaces} is focused on $\RPMs$ without the structure of a statistical mechanics model but with varying face degrees which are in the domain of attraction of an $\alpha$-stable random variable.  In the regime considered in \cite{lgm2011largefaces}, it is expected that the scaling limit can be described by an instance of the conformal loop ensemble ($\CLE$) \cite{s2009cle,sw2012cle}, the loop version of $\SLE$, on LQG, as studied in \cite{msw2020simplecle,msw2020nonsimplecle}.  $\CLE_\kappa$ corresponds to $\SLE_\kappa(\kappa-6)$, so the behavior is different from what we will encounter here.  We also mention \cite{ccm2020cascade,cr2020duality,budd2018peeling,bbck2018martingales,bbg2012recursive} as a sampling of other works which study $\RPMs$ which are expected to be in this regime.  Finally, let us also mention \cite{bco2019shredded}, which studies $\RPMs$ with large faces in the case of so-called causal maps and shows that in the scaling limit they belong to a universality class which is not related to $\SLE$ and LQG.

\subsection{Main results}

We now turn to state our main results.  Let us begin by describing the setting in a slightly more precise manner.
\begin{figure}[h!]
\centering
\includegraphics[scale=0.8]{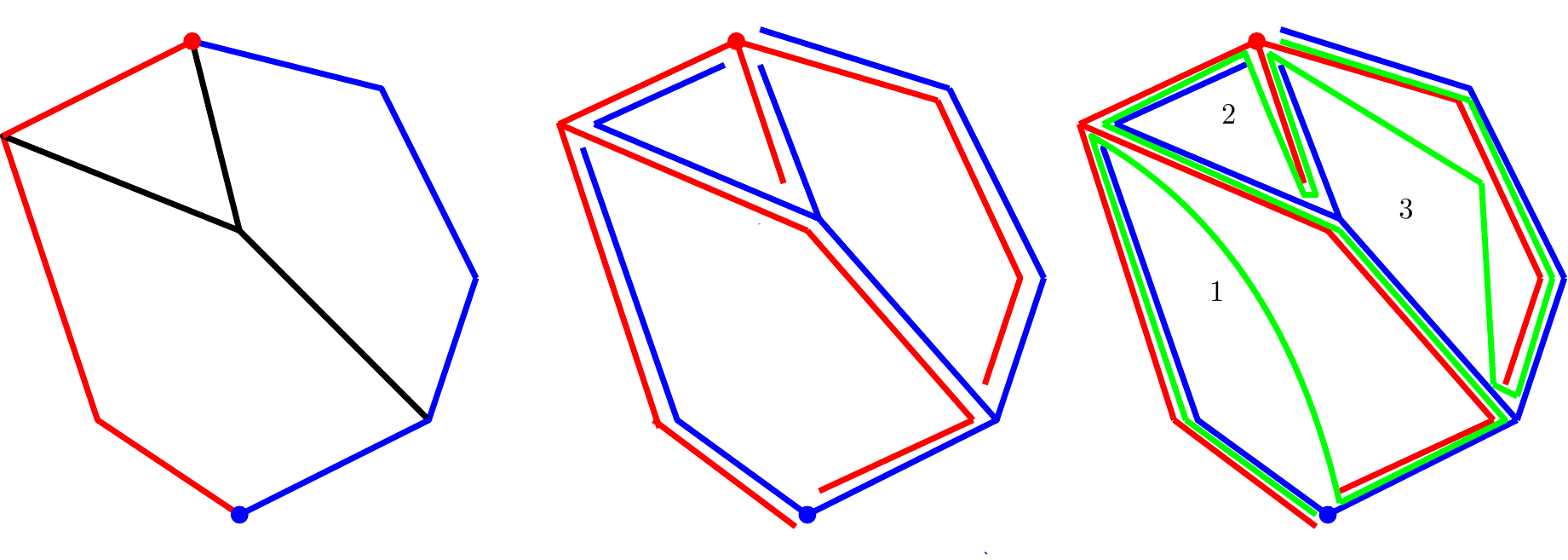}
\caption{\textbf{Left:} A bipolar oriented planar map,  embedded so that each edge is oriented in the direction along which the vertical coordinate increases.  The red (resp.\ blue) edges form the western (resp.\ eastern) boundary of the map.  Also the \emph{source} (resp.\ \emph{sink}) of the map is colored blue (resp.\ red). \textbf{Middle:} The set of oriented edges forms the \emph{northwest tree}, where the parent of each edge $e$ is the leftmost upward oriented edge emanating from the terminal vertex of $e$.  If we reverse the orientations of all edges, we can construct an analogous tree (blue) and embed both trees so that they do not cross.  \textbf{Right:} We trace the Peano curve between the two trees using a green path.  Each edge of the green path either moves along an edge of the map or across a face of the map.  The faces are numbered by the order they are traversed by the green path but we emphasize that the traversals of the edges of the green path are what correspond to steps of the lattice path.}
\label{fig:encoding_trees}
\end{figure}

Let $G$ be a bipolar oriented planar map with a finite number of edges.  We assume that the source and sink are incident to the same face and the planar map is embedded into the plane so that this is the unbounded face.  The \emph{western} (resp.\ \emph{eastern}) boundary consists of the edges incident to the unbounded face in the clockwise (resp.\ counterclockwise) boundary arc starting from the source and ending at the sink.  We consider the Peano curve associated with $G$ which winds between the $\NW$ and $\SE$ tree starting from the source (south pole) and ending at the sink (north pole). The Peano curve traverses in each step either an edge or an interior face.  Let~$E$ be the set of edges of~$G$, which we order $e_{0},\ldots,e_{|E|-1}$ according to when they are visited by the Peano curve. Then, if the western boundary has $\ell+1$ edges and the eastern boundary has $k+1$ edges, we can define a lattice path $(X_m , Y_m)_{0 \leq m \leq |E| - 1}$ starting at $(0,\ell)$ and ending at $(k,0)$ such that $X_m$ (resp.\ $Y_m$) is the distance in the $\SE$ (resp.\ $\NW$) tree from the source (resp.\ sink) to the lower (resp.\ upper) endpoint of $e_k$. It was shown in~\cite{kmsw2019bipolar} that there is a bijection between bipolar oriented planar maps with $n$ edges and $\ell+1$ (resp.\ $k+1$) ends in the western (resp.\ eastern) outer boundary and length-$(n-1)$ paths from $(0,\ell)$ to $(k,0)$ in the non-negative quadrant $\N_0^2$ having increments $(1,-1)$ and $(-i,j)$ with $i,j \geq 0$; we will give a review of this bijection in Section~\ref{sec:Bipolar}. We note that steps of the form $(1,-1)$ correspond to the case where the Peano curve traverses an edge while steps of the form $(-i,j)$ correspond to the case where the Peano curve traverses an interior face with $i+1$ edges on its west and $j+1$ edges on its east. In the present paper we will construct an infinite volume analog of the aforementioned bijection. In particular, we will construct a rooted $\RPM$ $(G,\rho)$ with infinite number of edges encoded by a lattice path $(X_m , Y_m)_{m \in \Z}$ and such that each increment $(X_{m+1}-X_m , Y_{m+1} - Y_m)$ will correspond to a step $w_m$ in the construction of $G$.  We note that in the infinite volume setting the lattice path does not stay in the non-negative quadrant.

Next we describe the probability measures on bipolar oriented planar maps that we are going to consider.  For fixed non-negative weights $(\weight_k)_{k \geq 2}$ such that at least one of them is positive,  we weight a bipolar oriented planar map by $\prod_{k=2}^{\infty}\weight_{k}^{n_k}$ where $n_k$ is the number of faces with $k$ edges. We also use the convention $0^{0} = 1$. Then we can pick a bipolar oriented planar map with $n$ edges with probability proportional to its weight.  We note that this defines a probability measure if at least one bipolar map has positive weight.  In Section~\ref{sec:Bipolar}, we will show that for every $\alpha \in (1,2)$, if we take the weights $(\weight_k)_{k \geq 2}$ to be such that $\weight_k = C_0 L^{-k} k^{-\alpha-2}$ for $C_0, L > 0$ then the increments of the random walk $(X,Y)$ encoding the bipolar planar map take the value $(-i,j)$ with probability $C_1 (i+j+2)^{-\alpha - 2}$ for $i,j \geq 0$, where $C_1 > 0$ is a constant depending only on $\alpha$.  We note that the choice of $L$ in fact does not matter because its total contribution to the weight of a given map with~$n$ edges is $L^{-n}$, hence gets absorbed into the normalization constant to get a probability measure.  In that case, the step distribution of $(X,Y)$ belongs to the domain of attraction of a two-dimensional $\alpha$-stable L\'evy process. For the rest of the paper, we will consider the aforementioned choice of weights.

We now state the notion of convergence of graphs that we are going to use. Recall that a rooted graph is a pair $(G,\rho)$ where $G$ is a connected graph and $\rho$ is a vertex in $G$.  Moreover, a rooted graph $(G,\rho)$ is isomorphic to $(G' , \rho')$ if there is a graph isomorphism from $G$ to $G'$ which takes~$\rho$ to~$\rho'$.  We will use the notation $B_G(\rho,r)$ to denote the graph metric ball centered at $\rho$ of radius $r$.

\begin{definition}
\label{def:graph_convergence}
Let $(G,\rho)$ and $(G_n,\rho_n)_{n \in \N}$ be random connected rooted graphs. We say that $(G,\rho)$ is the distributional limit of $(G_{n},\rho_{n})$ as $n \to \infty$ (or that $(G_{n},\rho_{n})$ converges to $(G,\rho)$ as $n \to \infty$ with respect to the Benjamini-Schramm topology) if for every $r>0$ and for every finite rooted graph $(H,\rho')$, the probability that $(H,\rho')$ is isomorphic to $(B_{G_{n}}(\rho_{n},r),\rho_{n})$ converges to the probability that $(H,\rho')$ is isomorphic to $(B_{G}(\rho,r),\rho)$.
\end{definition}

Our first main result is the existence of the Benjamini-Schramm (i.e., ``local'') limit for a $\BPRPM$ with large faces. More precisely, we fix $A,B > 0$ and for $n \in \N$ we let $G_n$ be a $\BPRPM$ with $n+1$ edges and conditioned such that it has $\lfloor A n^{1/ \alpha} \rfloor + 1$ (resp.\ $\lfloor B n^{1 / \alpha} \rfloor + 1$) edges on its western (resp.\ eastern) boundary. Note that we can root each $G_n$ by picking $\rho_n$ independently and uniformly among the vertices of $G_n$. Let also $(G,\rho)$ be the infinite volume $\RPM$ encoded by the lattice path $(X_m , Y_m)_{m \in \Z}$. Then the following holds.

\begin{theorem}
\label{thm:benjamini_schramm_conv}
The infinite volume rooted $\RPM$ $(G,\rho)$ is the distributional limit of the rooted $\BPRPMs$ $(G_n,\rho_n)$ as $n \to \infty$.
\end{theorem}

Our second main result is the convergence of a scaled version of the lattice path $(X,Y)$ used to construct the infinite volume $\RPM$ $G$ towards a coupled pair of $\alpha$-stable L\'evy processes.  Before we state our result, we review the definition of the Skorokhod topology.  For a compact interval $[a,b] \subseteq \R$,  we let $D([a,b])$ be the space of c\`adl\`ag functions $[a,b] \rightarrow \R$.  Let also $\Lambda_{a,b}$ be the space of continuous and increasing functions mapping $[a,b]$ onto $[a,b]$, and for $\lambda \in \Lambda_{a,b}$  we set 
\begin{align*}
\lVert \lambda \rVert^{0} = \sup_{a \leq s < t \leq b} \left | \log\left ( \frac{\lambda(t) - \lambda(s)}{t-s}\right ) \right|.
\end{align*}
Note that $D([a,b])$ becomes a separable and complete metric space under the metric
\begin{align*}
d_{a,b}(f,g) = \inf_{\lambda \in \Lambda_{a,b}} \left \{\max\{ \lVert \lambda \rVert^{0},  \lVert f-g \circ \lambda \rVert \} \right \}
\end{align*}
for $f,g \in D([a,b])$,  where $\lVert \cdot \rVert$ denotes the supremum norm on functions $[a,b] \to \R$.  Moreover,  the space $D(\R)$ endowed with the metric $d$ is  defined by 
\begin{align*}
d(f,g) = \sum_{n \in \N} 2^{-n} \min \{1, d_{-n,n}(f,g)\}.
\end{align*}
In this work, we will also consider random variables taking values in $D^2$ which we endow with the metric defined as above but with $\|\cdot\|$ in the term $\|f - g \circ \lambda\|$ in the definition of $d_{a,b}$ taken to be given by the supremum norm on functions $[a,b] \to \R^2$.

We will abuse notation and write
\begin{align*}
X_{nt}= X_{\lfloor nt \rfloor} \quad\text{and}\quad  Y_{nt} = Y_{\lfloor nt \rfloor} \quad\text{for}\quad  n \in \N,\  t \in \R.
\end{align*}
Note that the process $(X_{nt},Y_{nt})$ for $n \in \N$ and  $t \in \R$ is well-defined since $(X,Y)$ is indexed by $\Z$.

We recall from \cite{bertoin1996levy} also that that if $(t_k,j_k)$ are the jump time/magnitude pairs for an $\alpha$-stable L\'evy process $Z$ with $\alpha \in (1,2)$ then 
\begin{equation}
\label{eqn:levy_from_jumps}
Z_t = \lim_{\delta \to 0} \left(\sum_{t_k \leq t} j_k \one_{|j_k| \geq \delta} - \E\left[ \sum_{t_k \leq t} j_k \one_{|j_k| \geq \delta} \right] \right).	
\end{equation}

We are now ready to state our next result.

\begin{theorem}
\label{thm:boundary_length_conv}
The process $(n^{-1/\alpha}(X_{nt},Y_{nt}))_{t \in \R}$ converges in distribution with respect to the Skorokhod topology on $D^2 = D(\R) \times D(\R)$ as $n \to \infty$ to a correlated pair $W = (W^1 , W^2)$ of $\alpha$-stable L\'evy processes indexed by $\R$ whose law can be sampled from as follows.  Let $Z$ be a bi-infinite $\alpha$-stable L\'evy process with only upward jumps normalized so that $Z_0 = 0$ and whose L\'evy measure $\Pi$ is given by $\Pi(dx) = C_1 x^{-\alpha-1}\one_{x > 0}dx$ where $C_1 = C_1(\alpha) > 0$ is a constant depending only on $\alpha$.  Let $(t_k,j_k)_{k \in \N}$ be the jump time/magnitude pairs for $Z$ and let $(U_k)_{k \in \N}$ be an independent sequence of i.i.d.\ random variables which are uniform in $[0,1]$.  Then $W^1$ (resp.\ $W^2$) is the $\alpha$-stable L\'evy process associated with the jump time/magnitude pairs $(t_k,-U_k j_k)_{k \in \N}$ (resp.\ $(t_k,(1-U_k) j_k)_{k \in \N}$) as in~\eqref{eqn:levy_from_jumps}.
\end{theorem}
In the statement of Theorem~\ref{thm:boundary_length_conv}, the jumps associated with the process $Z$ correspond to the scaling limit of the overall face sizes in the $\BPRPM$ and the jumps associated with $W^1$ (resp.\ $W^2$) correspond to the scaling limit of the part of the face boundary length which is on the western (resp.\ eastern) boundary.

We will show in a companion work that one can view an $\SLE_\kappa(\kappa-4)$ process drawn on top of a certain type of independent LQG surface as the Peano curve which snakes between a pair of $\alpha$-stable looptrees which are constructed from the coordinate functions of a two-dimensional $\alpha$-stable L\'evy process which has the same law as in Theorem~\ref{thm:boundary_length_conv}.  Here, the relationship between $\alpha \in (1,2)$ and $\kappa \in (4/3,2)$ is given by
\[ \alpha = \frac{4}{\kappa}-1.\]
Note that for $\kappa=4/3$ we have that $\alpha=2$.  The reason for this is that the $\SLE_\kappa(\rho)$ processes are defined for $\rho > -2-\kappa/2$ and, as $\rho \to -2-\kappa/2$, an $\SLE_\kappa(\rho)$ process converges with $\kappa \in (0,4)$ to an $\SLE_{16/\kappa}$ process.  In particular, $\kappa=4/3$ corresponds to $\SLE_{12}$, which is natural to expect because a stable L\'evy process with $\alpha=2$ corresponds to Brownian motion.  At the other extreme, for $\kappa=2$ we get $\alpha=1$.  The reason for this is that at $\kappa=2$ we have that $\rho=\kappa-4 = -2$ and $-2$ is the threshold which separates the exotic and standard $\SLE_\kappa(\rho)$ processes (which by the theory developed in \cite{dms2014mating} are related to stable L\'evy processes with $\alpha \in (0,1)$).

We emphasize that the jumps of the two coordinates of the limiting process $W$ a.s.\ occur simultaneously.  This is in contrast to the stable L\'evy processes which arise when one considers the scaling limits of other types of random planar map models, i.e., a uniform quadrangulation decorated by a percolation configuration \cite{ck2015looptree,gm2017percolation}, where the coordinates are independent. 

Let us now explain where the value $\rho=\kappa-4$ arises.  It turns out that the complementary components of an $\SLE_\kappa(\kappa-4)$ process drawn on top of an appropriate independent $\sqrt{\kappa}$-LQG surface parameterize conditionally independent quantum disks.  In this case, the quantum disk is marked by two points corresponding to the first and last points on its boundary which are visited by the $\SLE_\kappa$-like excursion of the $\SLE_\kappa(\kappa-4)$ which makes up its left boundary (in the case that the force point is on the right side).  It turns out that the conditional law of these marked points given the surface are that of independent picks from the quantum disk boundary measure \cite{dms2014mating}.  As we will see later on in this work, each face in a $\BPRPM$ is naturally marked by two vertices (its ``north'' and ``south'') which are also in a certain sense uniformly random points on the face boundary given the boundary length of the face.

The type of scaling limit result considered in Theorem~\ref{thm:boundary_length_conv} is of what is now known as the ``peanosphere'' type, which means one takes a scaling limit of the contour functions which encode a pair of trees which can be glued together in a certain way to construct the planar map.  Other examples of this type include the previous work on $\BPRPMs$ \cite{kmsw2019bipolar} as well as \cite{s2016qginv,lsw2017schnyder,gkmw2018active}.  This type of scaling limit is different from that which encodes the overall metric structure of a RPM.  Limits of this type were first established by Le Gall \cite{lg2013bm} and Miermont \cite{m2013bm} and in this case the scaling limit is the so-called Brownian map.  Later works include the joint scaling limit results for self-avoiding walks \cite{gm2021saw} and percolation \cite{gm2017percolation} on random planar maps viewed as path-decorated metric spaces.

\subsection{Outline}

The remainder of this article is structured as follows.  We collect some preliminaries in Section~\ref{sec:Bipolar}. We prove the convergence of the boundary length processes in Section~\ref{sec:scaling_limit}. We then construct the infinite volume limit and prove the Benjamini-Schramm convergence in Sections~\ref{sec:total_variation} and~\ref{sec:B-S_convergence}.

\subsection*{Acknowledgements} K.K.'s work was supported by the EPSRC grant EP/L016516/1 for the University of Cambridge CDT (CCA).  J.M.'s work was supported by ERC starting grant 804166 (SPRS).  We thank Rick Kenyon, Scott Sheffield, and David Wilson for helpful discussions related to this work.

\section{Bipolar-oriented maps and lattice paths}
\label{sec:Bipolar}

\subsection{From lattice paths to bipolar maps}

We saw in Section~\ref{sec:introduction} that every bipolar map corresponds to a lattice path which remains in the non-negative quadrant by keeping track of the distances to the roots of the $\NW$ and $\SE$ trees as one visits the edges in the map using the ordering from the Peano curve. It is shown in~\cite{kmsw2019bipolar} that this construction can be reversed, constructing a bipolar oriented planar map from a lattice path of the above type.

The bipolar oriented planar map is constructed from the lattice walk by successively sewing edges and polygons depending on the sequence of increments of the lattice path. Following the terminology of~\cite{kmsw2019bipolar}, we denote by $\move_{i,j}$ a step where the walk increment is $(-i,j)$ with $i,j \geq 0$ and by $\move_{e}$ a step where the walk increment is $(1,-1)$. We note that in order to construct the infinite volume analog of the bipolar oriented planar map, it is convenient to extend the bijection so that it can be applied to lattice paths which do not necessarily remain in the non-negative quadrant. If we follow these steps, then we obtain a marked bipolar map, which is a bipolar oriented planar map but with some missing edges on its eastern and western boundary and with two fixed vertices lying on the eastern and western boundary respectively. The vertex which is on the western boundary is called the ``start vertex'' and it is not at the top of the western boundary and every vertex below it on the western boundary has at most one downward edge. On the other hand, the fixed vertex of the marked bipolar map which is on the eastern boundary is called the ``active vertex'' and it is not at the bottom and every vertex above it on the eastern boundary has at most one upward edge. The missing edges on the western (resp.\ eastern) boundary are those lying below (resp.\ above) the start (resp.\ active) vertex.

Next, to describe the construction of the marked bipolar map, we start with an edge whose lower endpoint is the start vertex and the upper endpoint is the active vertex. The $\move_{e}$ move will sew an edge to the current marked bipolar map such that the lower (resp.\ upper) endpoint of the new edge is the old (resp.\ new) active vertex. The upper endpoint of the new edge is added to the marked bipolar map if the eastern boundary does not have a vertex above the old active vertex. Otherwise, the new edge gets sewn to the southernmost missing edge on the eastern boundary of the current marked bipolar map. As for the $\move_{i,j}$ move, we sew an open face with $i+1$ edges on its west and $j+1$ edges on its east, where the old active vertex becomes the north of the face and the west of the face becomes part of the eastern boundary of the current marked bipolar map. An edge is also added to the southernmost east edge of the new face and the upper vertex of this edge becomes the new active vertex. It is worth mentioning that whenever there are fewer than $i+1$ edges below the old active vertex, then the start vertex is no longer at the bottom and the remaining edges of the face are missing from the map. The latter occurs when the encoding lattice path leaves the non-negative quadrant.

\begin{figure}[h!]
\centering
\includegraphics[scale=0.8]{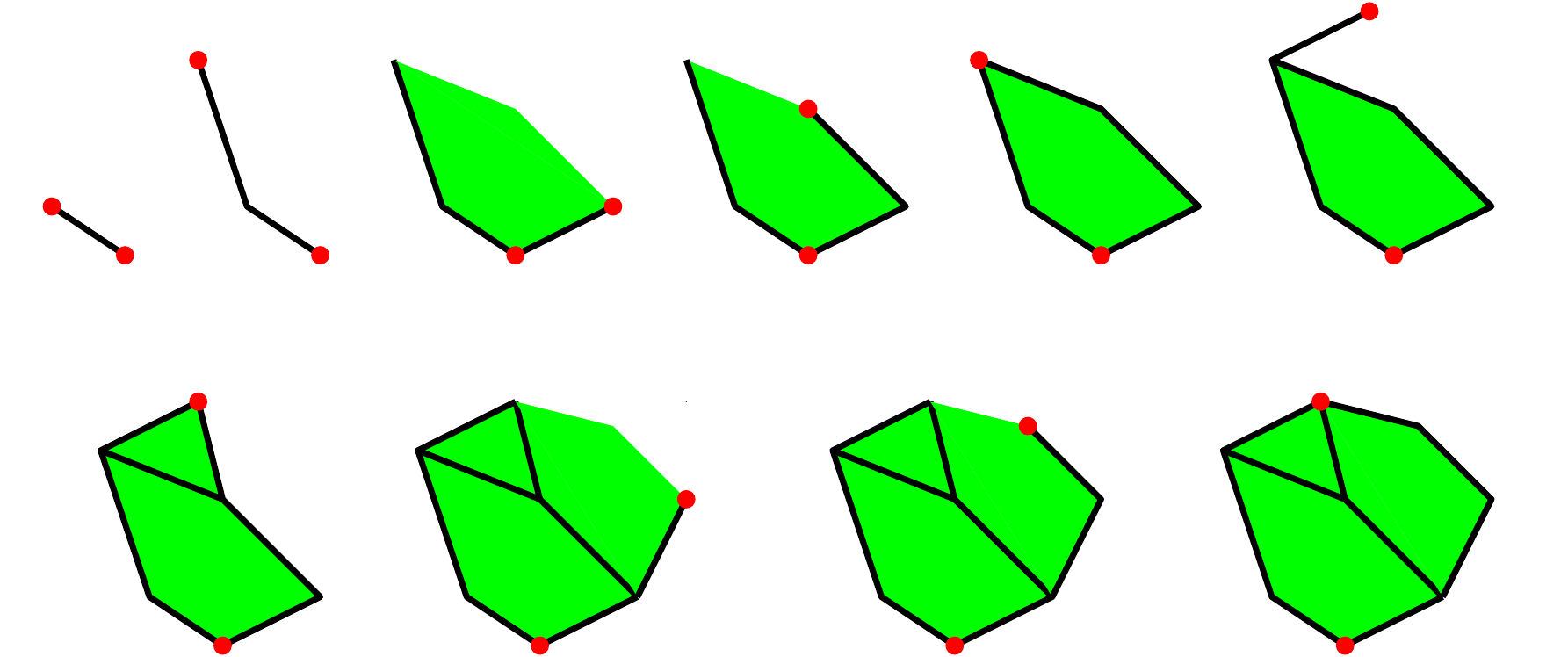}
\caption{:   The process of sewing oriented polygons and edges to obtain a bipolar oriented planar map.  The intermediate structures are marked bipolar oriented planar maps,  which may have some edges missing on the boundaries.  The sequence of steps is:
$\move_{e},\move_{1,2},\move_{e},\move_{e},\move_{e},\move_{1,0},\move_{1,2},\move_{e},\move_{e}$.}
\label{fig:sewing_faces_and_edges}
\end{figure}

The following result was proved in~\cite{kmsw2019bipolar}.

\begin{theorem} \label{thm:marked_bipolar_bijection}
The above mapping from sequences of moves from $\{\move_{e}\}\cup{\{\move_{i,j}: i,j \geq 0\}}$ to marked bipolar maps is a bijection.
\end{theorem}

Note that the final marked bipolar map is a bipolar oriented planar map if the start vertex is at the south and the active vertex is at the north. In~\cite{kmsw2019bipolar}, the following was also shown about the sequences of moves which give valid bipolar oriented planar maps:

\begin{theorem}
\label{thm:bipolar_maps_bijection}
The above mapping gives a bijection from length-$(n-1)$ paths from $(0,\ell - 1)$ to $(k-1,0)$ in the non-negative quadrant having increments $(1,-1)$ and $(-i,j)$ with $i,j \geq 0$, to bipolar-oriented planar maps with $n$ total edges and $\ell$ (resp.\ $k$) edges on the western (resp.\ eastern) boundaries.  A step of $(-i,j)$ in the walk corresponds to a face with degree $i+j+2$ in the planar map.
\end{theorem}

\subsection{Probability measures on bipolar maps}

Fix $\alpha \in (1,2)$. In this section,  we give certain conditions satisfied by the weights $(\weight_k)_{k \geq 2}$ as in Section~\ref{sec:introduction}, so that the step distribution of the random walk encoding the bipolar planar map belongs to the domain of attraction of a two-dimensional $\alpha$-stable L\'evy process.  Suppose that we fix non-negative weights $(\weight_k)_{k \geq 2}$ such that at least one of them is positive and we consider the probability measures on bipolar oriented planar maps with a finite number of edges corresponding to this choice of weights as in Section~\ref{sec:introduction}.  In other words,  we consider the probability measure on bipolar planar maps with $n$ total edges, $\ell$ edges on its western boundary,  and $k$ edges on its eastern boundary where the probability of each map is proportional to its total weight.  As in \cite{kmsw2019bipolar}, we assume that $\sum_{k=2}^\infty \weight_{k} z^{k}$ has a positive radius of convergence $R$ and 
\begin{align}
\label{eq:weight_ineq}
1 \leq \sum_{k =  2}^{\infty}\frac{(k-1)(k-2)}{2}\weight_{k}R^{k}.
\end{align}
The right-hand side of~\eqref{eq:weight_ineq} increases monotonically  from $0$ and is continuous on $[0,R)$. This implies the existence of $\lambda \in (0,R]$ such that 
\begin{align}\label{eq:1}
1= \sum_{k= 2}^{\infty}\frac{(k-1)(k-2)}{2}\weight_{k}\lambda^{k}.
\end{align}
Since $\weight_{k}>0$ for some $k \geq 3$, we have that $\lambda$ is finite.

Next let $\weight_{0}= 1$ and define
\begin{align}
\label{eq:2}
C= \frac{1}{\lambda^{2}}+\sum_{k= 2}^{\infty}(k-1)\weight_{k}\lambda^{k-2},
\end{align}
which by our hypothesis is finite, and define $p_{k}= \weight_{k}\lambda^{k-2}/C$ for $k \geq 0$. Then the $p_{k}$'s define a random walk $(X_m,Y_m)$ in $\Z^{2}$, which assigns probabilities $p_{0}$ and $p_{i+j+2}$ to steps $\move_{e}$ and $\move_{i,j}$ respectively and note that $p_1 = 0$ (we remind the reader that there are $k-1$ possible steps of type $\move_{i,j}$ where $i+j=k-2$, corresponding to a $k$-gon).  It is also shown in~\cite{kmsw2019bipolar} that the random walk with step distribution defined by the $p_k$'s has zero mean.

\begin{lemma}\label{lem:condition_on_a_k}
Fix $\alpha \in (1,2)$. Then it holds that $p_k = C_1 k^{-\alpha - 2}$ for all $k \geq 2$ for some finite and positive constant $C_1$ if and only if $\weight_k = C_0 L^{-k} k^{-\alpha - 2}$ for all $k \geq 2$ for some $L > 0$ and
\[ C_0 = \left(\sum_{k=2}^{\infty}\frac{(k-1)(k-2)}{2k^{\alpha + 2}}\right)^{-1}.\]
Moreover we have that 
\begin{align*}
C_1  = C_0 \left(1 + C_0 \sum_{k=2}^{\infty} (k-1)k^{-\alpha - 2}\right)^{-1}.
\end{align*}
\end{lemma}
\begin{proof}
Suppose there exists a constant $C_1 > 0$ so that $p_{k}= C_1 k^{-\alpha - 2}$ for each $k \geq 2$. Then, $\weight_{k}= \frac{\lambda^{2}C C_1}{\lambda^{k}k^{\alpha+2}}$ for each $k \geq 2$ and $R^{-1}= \limsup_{k \to \infty}\weight_{k}^{1/k}= \lambda^{-1} = L^{-1}$, where $C$ is defined as in~\eqref{eq:2}.  Hence~\eqref{eq:1} implies that
\begin{align*}
1 = \lambda^2 C C_1 \sum_{k=2}^{\infty}\frac{(k-1)(k-2)}{2k^{\alpha + 2}} = \frac{\lambda^2 C C_1}{C_0}
\end{align*}
and so 
\begin{align}\label{eq:3}
\lambda^2 C = \frac{C_0}{C_1}.
\end{align}
Moreover~\eqref{eq:2} implies that
\begin{align*}
C &= \frac{1}{\lambda^2} + C \sum_{k=2}^{\infty}(k-1)p_k
 =\frac{1}{\lambda^2} + C C_1 \sum_{k=2}^{\infty}\frac{k-1}{k^{\alpha + 2}}
\end{align*}
and so using~\eqref{eq:3} in the second equality we have that
\begin{align}
\label{eq:4}
\lambda^2 C &=1 + \lambda^{2} C C_1 \sum_{k=2}^{\infty}\frac{k-1}{k^{\alpha + 2}}
 =1 + C_0 \sum_{k=2}^{\infty}\frac{k-1}{k^{\alpha + 2}}.
\end{align}
By~\eqref{eq:3} and~\eqref{eq:4} we obtain that
\begin{align*}
\frac{C_0}{C_1} = 1 + C_0 \sum_{k=2}^{\infty} \frac{k-1}{k^{\alpha + 2}}
\end{align*}
and so rearranging gives that $C_1 = C_1(\alpha)$.

Conversely, suppose that $\weight_{k}= C_{0} L^{-k}k^{-\alpha-2}$ for all $k \geq 2$ and for some constant $L > 0$, where $C_0$ is as in the lemma statement. Then the definition of $C_0$ implies that
\begin{align*}
\sum_{k=2}^{\infty}\frac{(k-1)(k-2)}{2}\weight_k L^k = 1
\end{align*}
and so $L = \lambda$ and $p_k = C_0L^{-2} C^{-1} k^{-\alpha - 2}$ for all $k \geq 2$. Thus it suffices to show that $\frac{C_0}{L^2 C} = C_1(\alpha)$. But this follows by~\eqref{eq:2}, the definition of $C_0$ and the fact that $L = \lambda$.
\end{proof}

\section{Path scaling limit}
\label{sec:scaling_limit}

We will now give the proof of Theorem~\ref{thm:boundary_length_conv}.  Some parts of the argument that we will give are based on \cite{rg1979bivariate}.  Fix $\alpha \in (1,2)$ and let $p_0$, $p_k$ for $k \geq 2$ be as in Lemma~\ref{lem:condition_on_a_k}. Let $(X_{n},Y_{n})_{n\in \N}$ be the lattice walk in $\Z^{2}$ starting from $(0,0)$ with i.i.d.\ increments satisfying $\p[(X_{n}-X_{n-1},Y_{n}-Y_{n-1})= (1,-1)]= p_{0}$ and $\p[(X_{n}-X_{n-1},Y_{n}-Y_{n-1})=(-i,j)]= p_{i+j+2}$ for each $ n \in \N$ and $ i,j\in \N_{0}$.  Also, for $1\leq k \leq n$
we define 
\begin{align*}
W_{k,n}= n^{-1/\alpha} (X_{k}-X_{k-1},Y_{k}-Y_{k-1}).
\end{align*}
Consider the Radon measure $\nu$ on $\R^{2}\setminus \{0\}$ where 
\begin{align*}
d\nu(x,y)= C_1 \one_{(-\infty,0)}(x) \one_{(0,\infty)}(y)(-x+y)^{-\alpha-2}dxdy
\end{align*}
with $dxdy$ denoting Lebesgue measure on $\R^2$ and $C_1$ is as in Lemma~\ref{lem:condition_on_a_k}.  Let $\wt{W} = (\wt{W}^1,\wt{W}^2)$ be the L\'evy process in $\R^{2}$ starting from $(0,0)$ and with characteristic exponent $\psi$ given by 
\begin{align*}
\int_{\R^{2}}(1-e^{i(\lambda,x)}+i(\lambda,x)\one_{|x| \leq 1}) d\nu(x) \quad\text{for each} \quad  \lambda \in \R^{2}.
\end{align*}
We also define $\wt{W}_{t}$ for $t \leq 0$ by letting $W'$ be an independent copy of $\wt{W}$ and then setting $\wt{W}_t = - W_{-t}'$ for $t \leq 0$.  We obviously have that $\nu$ is the corresponding L\'evy measure of $\wt{W}_{t}$.  Recall that if $\overline{\R^d}$ is the one-point compactification of $\R^d$,  then a sequence $(\nu_{n})$ of Radon measures on $\overline{\R^d} \setminus \{0\}$ converges vaguely to a Radon measure $\nu$ on $\overline{\R^d} \setminus \{0\}$ if 
\begin{align*}
\lim_{n \to \infty}\nu_{n}(B)= \nu(B)
\end{align*}
for every Borel subset $B$ of $\overline{\R^{d}} \setminus \{0\}$ such that $0 \notin \overline{B}$, and  $\nu(\partial B) = 0$ or equivalently 
\begin{align*}
\lim_{n \to \infty}\int_{\R^{d}}f(x) d\nu_{n}(x)=\int_{\R^{d}}f(x) d\nu(x)
\end{align*}
for every continuous function $f$ on $\overline{\R^{d}} \setminus \{0\}$ whose support is contained in a Borel subset $B $ of $\overline{\R^{d}} \setminus \{0\}$ with $0 \notin \overline{B}$.

We are now ready to prove Theorem~\ref{thm:boundary_length_conv}.

\begin{proof}[Proof of Theorem~\ref{thm:boundary_length_conv}]
We will prove that $(n^{-1/\alpha} (X_{nt},Y_{nt}))_{t \in \R}$ converges in distribution  as $n \to \infty$ to $W_t = \wt{W}_{t}+p(\alpha)t$ with respect to the Skorokhod topology on $D^2 = D(\R) \times D(\R)$ where 
\begin{align*}
m(\alpha) = \frac{C_1}{(\alpha - 1)(\alpha + 1)} + \frac{C_1}{\alpha + 1} \int_{0}^1 \frac{x dx}{(x + \sqrt{1-x^2})^{\alpha + 1}} \quad \text{and} \quad p(\alpha) = (m(\alpha),-m(\alpha)).
\end{align*}
Note that $p(\alpha) = -\E[\wt{W}_1]$ and that $\wt{W}_t + p(\alpha)t$ is a martingale.  To show this,  it suffices to prove that the convergence holds on $D([0,\infty)) \times D([0,\infty))$ when both of the processes are indexed by $[0,\infty)$.

Set 
\begin{align*}
Z_{n}(t)= \sum_{1 \leq i \leq nt}W_{i,n} - \sum_{1 \leq i \leq nt}\E[W_{i,n}\one_{ |W_{i,n}|\leq 1}],\quad n \in \N,\quad t \geq 0
\end{align*}
and let $\nu_{n}$ be the Radon measure defined by $\nu_{n}(B)= n\p[W_{1,n}\in B]$ for every Borel set $B$ in $\R^{2}\setminus \{0\}$. For $\delta \in (0,1)$ we also set 
\begin{align*}
&Z_{n,\delta}(t) = \sum_{1 \leq i \leq nt}W_{i,n}\one_{|W_{i,n}| \geq \delta} - \sum_{1 \leq i \leq nt} \E[W_{i,n} \one_{|W_{i,n}| \in (\delta,1)}] \quad\text{and}\\
&Z_{\delta}(t) = \sum_{t_{k} \leq t} j_{k} \one_{|j_k| \geq \delta} -t \int_{|u| \in (\delta,1)} u d\nu(u)
\end{align*}
where $(t_k,j_k)_{k \in \N}$ are the jump time/magnitude pairs of $\wt{W}$.  Note that
\begin{align*}
\wt{W}_t = \sum_{t_k \leq t} j_k \one_{|j_k| \geq 1} + \lim_{\delta \to 0}\!\left( \sum_{t_k \leq t} \one_{|j_k| \in (\delta , 1)} - \E\!\left[ \sum_{t_k \leq t} j_k \one_{|j_k| \in (\delta , 1)} \right] \right)
\end{align*}
and the limit is almost sure and uniform on compact time intervals \cite{bertoin1996levy}. We will deduce the claim in three steps.  In {\it Step 1}, we show that $\nu_n$ converges vaguely to $\nu$ as $n \to \infty$. In {\it Step 2}, we  prove the weak convergence of $Z_{n,\delta}$ to $Z_{\delta}$ as $n \to \infty$ in $D^2$ using {\it Step 1} for every fixed $\delta \in (0,1)$. Finally, in {\it Step 3}, we show that Prokhorov distance between $Z_{n,\delta}$ and $Z_{n}$ in $D^2$ on compact time intervals can be made arbitrarily small when $\delta > 0$ is small, uniformly in $n$. We then conclude the proof by the weak convergence of~$Z_{\delta}$ to~$W$.

\noindent{\it Step 1}. Firstly, we claim that the sequence of measures $\nu_{n}$ converges vaguely  to $\nu$ as $n \to \infty$. Indeed, it suffices to show that for every Borel set $B$ in $\R^{2} \setminus \{0\}$ such that $0 \notin \overline{B}$ and $\nu(\partial{B})= 0$, we have that $\nu_n(B) \to \nu(B)$ as $n \to \infty$. It is easy to verify that $\nu_{n}(B)$ converges to $\nu(B)$ as $n \to \infty$ if $B$ is a rectangle (either open or closed). Now, let $U$ be an open subset of $\R^{2}$ such that $0 \notin \overline{U}$. Then $U$ can be written as a countable union of open rectangles $(x,y)\times (z,w)$ such that either $0 \notin [x,y]$ or $0 \notin [z,w]$. Hence, there exists a sequence of Borel  sets $I_{n}$ with $I_{n}\subseteq I_{n+1}$ such that $0 \notin \overline{I_n}$, $U= \cup{I_{n}}$ and $I_{n}$ is a finite union of rectangles $I_{n,1},\ldots,I_{n,k_{n}}$ whose interiors are pairwise disjoint. Therefore, we obtain that 
\begin{align*}
    \nu(U) = \lim_{n \to \infty}\nu(I_{n})= \lim_{n \to \infty}\sum_{j=1}^{k_{n}}\nu(I_{n,j}) \quad\text{and}\quad
    \nu(I_{n}) = \lim_{m \to \infty}\sum_{j= 1}^{k_{n}}\nu_{m}(I_{n,j})= \lim_{m \to \infty}\nu_{m}(I_{n}).
\end{align*}
Thus as $\nu_m(I_n) \leq \nu_m(U)$ we see that
\begin{align*}
   \nu(I_{n})\leq \liminf\limits_{m \to \infty}\nu_m(U) \quad\text{for all}\quad n \in \N.
\end{align*}
By letting $n \to \infty$, we obtain that 
\begin{align}\label{eq:5}
\nu(U)\leq \liminf\limits_{m \to \infty}\nu_{m}(U).
\end{align}
Now, let $B$ be a Borel subset of $\R^{2} \setminus \{0\}$ such that $0 \notin \overline{B}$  and $\nu(\partial{B})= 0$. Then there exists $r>0$ such that $B \subseteq{U}$ where $U = \R^2 \setminus [-r,r]^2$.  Note that $\nu(U) < \infty$.  Since $U$ is a finite union of pairwise disjoint rectangles, we have that 
\begin{align}\label{eq:6}
\nu(U)=\lim_{n \to \infty}\nu_n(U).
\end{align}
Then~\eqref{eq:5} and~\eqref{eq:6} together with $\nu(U) = \nu(\ol{U})$ imply that 
\begin{align*}
\nu(U)-\nu(\overline{B})\leq \liminf\limits_{n \to \infty}\nu_{n}(U \setminus \overline{B})= \nu(U)-\limsup_{n \to \infty}\nu_{n}(\overline{B}) 
\end{align*}
and hence with $B^\circ$ denoting the interior of $B$ we have
\begin{align*}
       \limsup_{n \to \infty}\nu_{n}(\overline{B})\leq \nu(B)= \nu(\overline{B}) \quad\text{and}\quad
      \nu(B)= \nu(B^{\circ})\leq \liminf\limits_{n \to \infty}\nu_{n}(B).
\end{align*}
Therefore, we have  that $\nu(B)= \lim_{n \to \infty}\nu_{n}(B)$ and since $B$ was arbitrary, we have shown the claim.

\noindent{\it Step 2}. The vague convergence of $(\nu_n)$ to $\nu$ combined with the continuous mapping theorem applied to the summation functional implies that
\begin{align*}
\left(t \mapsto \sum_{1 \leq i \leq n t}W_{i,n}\one_{|W_{i,n}| \geq \delta}\right) \to \left(t \mapsto \sum_{t_{k} \leq t}j_{k}\one_{ |j_{k}| \geq \delta} \right) \quad\text{as}\quad n \to \infty
\end{align*}
in distribution in $D^2$.  Also 
\begin{align*}
\lim_{n \to \infty}\E\!\left[\sum_{i= 1}^{n}W_{i,n}\one_{|W_{i,n}|\in (\delta,1)}\right] = \lim_{n \to \infty}n\E\!\left[W_{i,n}\one_{|W_{i,n}|\in (\delta,1)}\right]= \int_{|u|\in (\delta,1)}u d\nu(u)
\end{align*}
and hence we obtain for every fixed $\delta \in (0,1)$ that $Z_{n,\delta}$ converges in distribution to $Z_{\delta}$ as $n \to \infty$.

\noindent{\it Step 3}. Now we show that $Z_{n,\delta}$ and $Z_{n}$ on compact time intervals are close in the Prokhorov distance uniformly in~$n$, for $\delta > 0$ sufficiently small. Fix $\epsilon, \delta \in (0,1)$. Then if $\rho$ is the Shorokhod metric in $D^2$,  by restricting to $D([0,1]) \times D([0,1])$ we have that 
\begin{align}
 \p[\rho(Z_{n,\delta},Z_{n})>\epsilon]
 &\leq \p[\sup_{0 \leq t \leq 1}|Z_{\delta,n}(t) - Z_{n}(t)| > \epsilon] \nonumber\\
 &=\p\!\left[ \sup_{1 \leq k \leq n} \left| \sum_{i=1}^k W_{i,n}\one_{|W_{i,n}| \leq \delta} - \E[ W_{i,n} \one_{|W_{i,n}| \leq \delta}]\right| > \epsilon \right] \nonumber \\
&\leq \sum_{j= 1}^{2}\p\!\left[\sup_{1 \leq k \leq n}\left|\sum_{i= 1}^{k}W_{i,n}^{j}\one_{|W_{i,n}|\leq \delta}-\E[W_{i,n}^{j}\one_{|W_{i,n}|\leq \delta}]\right|>\frac{\epsilon}{2}\right] \nonumber\\
&\leq \sum_{j=1}^2 \frac{4}{\epsilon^2} \var\!\left[ \sum_{i=1}^n W_{i,n}^j \one_{|W_{i,n}| \leq \delta} \right] \quad\text{(by Kolmogorov's inequality)}\nonumber \\
& \leq \sum_{j= 1}^{2}\frac{4}{\epsilon^{2}}n\E[|W_{1,n}^{j}|^{2}\one_{|W_{1,n}|\leq \delta}] \leq \sum_{j=1}^2 \frac{4}{\epsilon^2} n \E[ |W_{1,n}^j |^2 \one_{|W_{1,n}^j| \leq \delta} ] \label{eq:8}.
\end{align}
Note that there is a finite constant $C$ depending only on $\alpha$ such that for all $0< \eta < \delta$ we have that
\begin{align*}
n \E[ |W_{1,n}^j |^2 \one_{|W_{1,n}^j| \leq \eta} ] \leq C \eta^{\frac{2}{\alpha}-1} \quad \text{for}\quad j=1,2
\end{align*}
for all $n$ sufficiently large.  Also the vague convergence of $(\nu_n)$ to $\nu$ implies that
\begin{align*}
n \E[ |W_{1,n}^j |^2 \one_{|W_{1,n}^j| \in (\eta,\delta]} ] \to \int_{\eta <|x| \leq \delta}|x|^{2} d\nu^{j}(x) \quad \text{for}\quad j=1,2
\end{align*}
as $n \to \infty$ where $\nu^{1}$, $\nu^{2}$ are the L\'evy measures of $\wt{W}^{1}$, $\wt{W}^{2}$  respectively.  Combining and since $\alpha \in (1,2)$,  we obtain that the right hand side of~\eqref{eq:8} converges to $\sum_{j= 1}^{2}\frac{4}{\epsilon^{2}}\int_{|x|\leq \delta}|x|^{2} d\nu^{j}(x)$.  Since $\int_{|x|\leq \delta}|x|^{2} d\nu^{j}(x)<\infty$ as $\delta \to 0$, we obtain that 
\begin{align}
\label{eq:9}
\lim_{\delta \to 0}\limsup_{n \to \infty}\p[\rho(Z_{n,\delta},Z_{n})>\epsilon)]=0.
\end{align}

Let
\begin{align*}
  a_n = n^{1/\alpha},\quad b_{n}= \E\left[X_{1}\one_{X_{1}^{2}+Y_{1}^{2}\leq a_n^2}\right], \quad\text{and}\quad
  c_{n}= \E\left[Y_{1}\one_{X_{1}^{2}+Y_{1}^{2}\leq a_n^2}\right].
\end{align*}
 Combining~\eqref{eq:9} with the convergence of $Z_{n,\delta}$ to $Z_{\delta}$ as $n \to \infty$ for every fixed $\delta \in (0,1)$ and the convergence of $Z_{\delta}$ to $\wt{W}$ as $\delta \to 0$, we obtain that
\begin{align*}
a_n^{-1} \left(X_{nt}-nt b_{n},Y_{nt}-ntc_{n}\right) \quad\text{for}\quad t \geq 0
\end{align*}
converges in distribution as $n \to \infty$ (with respect to the Skorokhod topology in $D^2$) to $\wt{W}$.  Therefore $(n^{-1/\alpha} (X_{nt},Y_{nt}))_{t \in \R}$  converges in distribution to $W_t$ as $n \to \infty$ since
\begin{align*}
 \lim_{n \to \infty}\frac{nb_{n}}{a_{n}}= m(\alpha) \quad\text{and}\quad \lim_{n \to \infty}\frac{nc_n}{a_n} = -m(\alpha).
\end{align*} 
Finally,  in order to complete the proof,  it remains to show that the process $W$ is the same as the one considered in the statement of the theorem.  To show this,  we let $(t_k,  j_k,  u_k)_{k \in \N}$ be a Poisson point process with intensity measure given by $\mu = \Pi \times m$ where $m$ is the restriction of the one-dimensional Lebesgue measure on $[0,1]$ and $\Pi(dx) = C_1 x^{-\alpha-1}\one_{x > 0} dx$.  We also consider the measurable function $F : \R_+ \times \R_+ \times [0,1] \mapsto \R_+ \times \R_- \times \R_+$ with $F(t,x,u) = (t,-ux,(1-u)x)$.  Then we have that 
\begin{align*}
(F(t_k,j_k,u_k))_{k \in \N} = ((t_k,-u_k j_k,  (1-u_k)j_k))_{k \in \N}
\end{align*}
is a Poisson point process with intensity measure $\wt{\mu} = \mu_{*}F$.  It is easy to see that $\wt{\mu}$ and $\nu$ agree on rectangles of the form $[a,b] \times [c,d]$ for $a<b<0<c<d$ and so this implies that $\wt{\mu}$ and $\nu$ coincide.  This completes the proof.
\end{proof}

\section{Local total variation convergence}
\label{sec:total_variation}

Fix $\alpha \in (1,2)$ and assume that we have chosen weights $(\weight_k)_{k \geq 2}$ satisfying the hypothesis of the statement of Lemma~\ref{lem:condition_on_a_k}.  Fix $A, B > 0$, let $A_n = \lfloor A n^{1/\alpha} \rfloor$, and $B_n = \lfloor B n^{1/\alpha} \rfloor$.  Let $G_n$ be a $\BPRPM$ with $n+1$ edges sampled with probability proportional to its weight and conditioned such that it has $A_n + 1$ (resp.\ $B_n + 1$) edges on its western (resp.\ eastern) boundary.  Note that Theorem~\ref{thm:bipolar_maps_bijection} and Lemma~\ref{lem:condition_on_a_k} imply that the law of $G_n$ can be viewed as the law on length-$n$ paths from $(0, A_n)$ to $(B_n,0)$ in $\N_0^2$ with increments in $S = \{(1,-1),\,(-i,j),\,i,j \in \N_0 \}$ and step distribution $\mu = \{p_0,\,p_{i+j+2},\,i,j \in \N_0 \}$. Moreover, let $w_0,\ldots w_{n-1}$ be the sequence of steps in $S$ which encodes $G_n$ as in Section~\ref{sec:Bipolar}. For $n \in \N$ we let $\mu_n$ be the uniform measure on lattice paths on $\Z^2$ of length $n$ whose step distribution is given by $\mu$.  Next, we fix $r \in \N$ and let $(U_n)$ be a sequence of independent random variables where~$U_n$ has the uniform distribution on the set $\{r,\ldots,n-r\}$ for $2r \leq n$. Consider the random variable $M_{n,2r+1} = (w_{U_{n} - r},\ldots,w_{U_{n}+r})$ and let $m_{n,2r+1}$ be its law. Then we have the following result for the asymptotic local structure of $G_n$ as $n \to \infty$.

\begin{theorem}
\label{thm:local_convergence}
 The sequence of measures $m_{n,2r+1}$ on $S^{2r+1}$ converges in total variation to $\mu_{2r+1}$ as $n \to \infty$.
 \end{theorem}
 
Before we prove Theorem~\ref{thm:local_convergence}, we describe the main steps of its proof.  We first note that it suffices to show that $m_{n,2r+1}(b) \to \mu_{2r+1}(b)$ as $n \to \infty$ for all $b \in S^{2r+1}$.  Fix $b \in S^{2r+1}$ and let
\begin{align*}
A_{m}^i = \{w \in S^{\N_0}:\,w(i+j) = b_{j - (m-1)(2r+1)}\ \text{for all}\ (m-1)(2r+1) \leq j \leq m(2r+1)-1 \}
\end{align*}
Let also $F_n$ be the event that $X_{k},Y_{k} \geq 0$ for all $1 \leq k \leq n$ and $(X_{n},Y_{n}) = (B_n, 0)$, $k_n =  \lfloor \frac{n-2r}{2r+1} \rfloor$, and $C_{k_n}^i = \sum_{m=1}^{k_n} \one_{A_{m}^i}$.  We will control the behavior of $m_{n,2r+1}(b)$ by estimating
\begin{align}
\label{eq:conditioning_average}
\frac{1}{2r+1} \sum_{i=0}^{2r} \E\!\left[ \frac{C_{k_n}^i}{k_n}\,\middle | \, F_n \right].
\end{align}
Since for each $0 \leq i \leq 2r$ the sequence of random variables $(\one_{A_{m}^i})_{m \in \N}$ is i.i.d.\ with mean $p = \mu_{2r+1}(b)$,  Cramer's theorem implies for fixed $\epsilon > 0$ that off an event whose probability decays exponentially in $n$ (without conditioning on $F_n$), $k_n^{-1} C_{k_n}^i \in (p-\epsilon,p+\epsilon)$. Therefore, by~\eqref{eq:conditioning_average} it suffices to show that $\p[F_n]$ decays to $0$ as $n \to \infty$ more slowly than an exponential in $n$. In what follows, we let $a_n = \lfloor n^{1/\alpha} \rfloor$ and we will sometimes use the notation $\p_{(x,y)}$ for the law under which $(X_0,Y_0) = (x,y)$ and $\E_{(x,y)}$ for the corresponding expectation.

\begin{proposition}
\label{proposition:lower_bound_conditioning}
There exists a constant $q \in (0,1)$ (depending on $A,B$) such that
\begin{align*}
 \p_{(0,A_n)}[ X_{k},Y_{k} \geq 0 \ \text{for all}\  1\leq k\leq n ,\  X_{n}= B_n,Y_{n}= 0] \geq q^{a_n} \quad\text{for all}\quad n \in \N.
\end{align*}
\end{proposition}

Let us now describe the strategy that we will use to prove Proposition~\ref{proposition:lower_bound_conditioning}.  Suppose that $(X_0 , Y_0) = (0,A_n)$.  In the proof, we will choose $A^1,A^2 > 0$ such that $A = A^1 + A^2$.  Let $A_n^1 = \lfloor A^1 n^{1/\alpha} \rfloor$ and $A_n^2 = A_n - A_n^1$.  Then with probability at least $p_0^{A_n^1}$ (by considering $A_n^1$ consecutive steps of the form $(1,-1)$), we have that the random walk stays in $\N_0^2$ for the first $A_n^1$ steps and $(X_{A_n^1}, Y_{A_n^1}) = (A_n^1 , A_n^2)$.  Fix $c \in (0, \min(A^1,A^2))$ and we let $C_n = \lfloor C n^{1/\alpha} \rfloor$ where $C >0$ is a large constant depending only on $A$, $B$, $A^1$, $A^2$, $c$.  Next, conditioned on the above, with positive probability (uniform in $n$ but depending on $c >0$) in the next $b_n = n - A_n^1 - C_n$ steps the random walk stays in $R_n = [-c a_n +A_n^1, c a_n + A_n^1] \times [-c a_n + A_n^2,c a_n  + A_n^2] \subseteq \N_0^2$.  Finally, we conclude the proof by showing that we can find a constant $q>0$ which is uniform on $n$ and $(x,y) \in R_n$ such that the following holds.  Conditioned on $(X_{n-C_n},Y_{n-C_n}) = (x,y)$, the random walk stays in $\N_0^2$ for the remaining $C_n$ steps and $(X_n,Y_n) = (0,B_n)$ with probability at least $q^{a_n}$.

We now proceed to show that for for each $c > 0$ the event that the random walk scaled by $a_n$ stays in a fixed square $[-c,c]^2$ during the first $n$ steps has positive probability (uniform in $n$).  We will subsequently prove Proposition~\ref{proposition:lower_bound_conditioning} and then finally give the proof of Theorem~\ref{thm:local_convergence}.

\begin{proposition}
\label{proposition:lower_bound}
For each $c > 0$ there exists $p > 0$ so that
\begin{align*}
 \p_{(0,0)}\!\left[a_n^{-1}(X_k,Y_k) \in [-c,c]^2 \ \text{for}\ 1 \leq k \leq n \right] \geq p \quad\text{for all}\quad n \in \N.
\end{align*}
\end{proposition}

Let $\wt{W} = (\wt{W}^1,\wt{W}^2)$ be the limiting process from Section~\ref{sec:scaling_limit} (see also Theorem~\ref{thm:boundary_length_conv}) and for $\epsilon >0 $ we set 
\begin{align*}
\wt{W}_{1}^{\epsilon}(t)= \mu_{\epsilon}t+N^{\epsilon}(t) \quad\text{and}\quad \wt{W}_{\epsilon}(t)= \wt{W}(t)-\wt{W}_{1}^{\epsilon}(t)
\end{align*}
where $\Delta \wt{W}(s)$ denotes the jump of $\wt{W}$ at time $s$ and
\begin{align*}
&\mu_{\epsilon}= -\int_{\epsilon \leq |x| \leq 1} x  d\nu(x) \quad\text{and}\quad N^{\epsilon}(t)= \sum_{0 \leq s \leq t}\Delta \wt{W}(s)\one_{|\Delta \wt{W}(s)| \geq \epsilon}.
\end{align*}
Let $\nu$ be the L\'evy measure for $\wt{W}$.  Then $\wt{W}_{\epsilon}$ has characteristic exponent
\begin{align*}
\psi_{\epsilon}(u)= \int_{\R^{2}}(1-e^{i(u,x)}+i(u,x))\one_{|x|\leq \epsilon} d\nu(x),\quad u \in \R^{2}
\end{align*}
and $\wt{W}_{1}^{\epsilon} , \wt{W}_{\epsilon}$ are independent. Let 
\begin{align*}
\sigma(\epsilon)^2 = \int_{|x|^{2}+|y|^{2}\leq \epsilon^{2}}(x^{2}+y^{2}) d\nu(x,y).
\end{align*}
By a change of variables, we observe that 
\begin{equation}
\label{eqn:sigma_form}
\sigma(\epsilon)= \epsilon^{1-\frac{\alpha}{2}} \sigma(1) \quad\text{for all}  \quad \epsilon \in (0,1).
\end{equation}

Now, we consider the process $Y_{\epsilon}(t)= \frac{\wt{W}_{\epsilon}(t)}{\sigma(\epsilon)}$ for $t\in [0,1]$. Note that $Y_{\epsilon}$ has characteristic exponent
\begin{align*}
\phi_{\epsilon}(u)= \int_{|x|\leq \epsilon}\left(1-e^{i\left(\frac{u}{\sigma(\epsilon)},x\right)}+i\left(\tfrac{u}{\sigma(\epsilon)},x\right)\right) d\nu(x)
\end{align*}
 and thus we obtain that $Y_{\epsilon}$ has L\'evy measure $\nu_{\epsilon}$ given by 
\begin{align*} 
d \nu_{\epsilon}(x,y)= C_1 \one_{(-\infty,0)}(x)\one_{(0,\infty)}(y)\one_{x^{2}+y^{2}\leq \frac{\epsilon^{2}}{\sigma(\epsilon)^{2}}}(x,y)\frac{\sigma(\epsilon)^{-\alpha}}{(-x+y)^{\alpha+2}} dxdy.
\end{align*}
We aim to prove that $Y_{\epsilon}$ has a non-trivial limit as $\epsilon \to 0$. For this reason, we mention the following result which gives necessary and sufficient conditions for weak convergence \cite[Theorem~13.14]{kallenberg1997foundations}.

\begin{theorem}
\label{thm:Kalenberg}
 Let $\mu=\id(\alpha,\beta,\nu)$, $\mu_{n}= \id(\alpha_{n},\beta_{n},\nu_{n})$ be a sequence of infinitely divisible laws on~$\R^{d}$ and fix any $0<h<1$ with $\nu({|x|= h})=0$. Then $\mu_{n}$ converges weakly to $\mu$ if and only if 
\begin{align*} 
\lim_{n \to \infty} \alpha_{n}^{h}= \alpha^{h},\quad \lim_{n \to \infty} \beta_{n}^{h}= \beta^{h}
\end{align*}
and $\nu_{n} \to \nu$ vaguely on $\overline{\R^d}\setminus \{0\}$ as $n \to \infty$ where 
\begin{align*}
\alpha^{h}= \alpha+\int_{|x|\leq h} xx^{T} d\nu(x) \quad\text{and}\quad \beta^{h}= \beta-\int_{h\leq |x| \leq 1}x d\nu(x).
\end{align*}
\end{theorem}

Now, to prove the weak convergence of $Y_{\epsilon}$, fix $0<h<1.$ Since $\lim_{\epsilon \to 0}\tfrac{\epsilon}{\sigma(\epsilon)}= 0$, we have for $\epsilon >0$ sufficiently small that $\epsilon/\sigma(\epsilon) \leq h$ hence (recall~\eqref{eqn:sigma_form})
\begin{align*}
\int_{x^{2}+y^{2}\leq h^{2}}(x^{2}+y^{2}) d\nu_{\epsilon}(x,y)=  \sigma(\epsilon/\sigma(\epsilon))^2 \cdot \sigma(\epsilon)^{-\alpha} = 1
\end{align*}
and hence by symmetry we obtain that 
\begin{align*}
\lim_{\epsilon \to 0}\int_{x^{2}+y^{2}\leq h^{2}}x^{2} d \nu_{\epsilon}(x,y)= \lim_{\epsilon \to 0}\int_{x^{2}+y^{2}\leq h^{2}}y^{2} d\nu_{\epsilon}(x,y)=\frac{1}{2}.
\end{align*}
We similarly have that
\begin{align*}
\lim_{\epsilon \to 0}\int_{x^{2}+y^{2}\leq h^{2}}xy d\nu_{\epsilon}(x,y)=  r \in [-1/2,0].
\end{align*}
Indeed, the reason that we have $r \geq -1/2$ is that $|xy| \leq x^2/2 + y^2/2$ and we have that $r \leq 0$ because under $\nu_\epsilon$ we a.e.\ have that $x < 0$ and $y > 0$.

Next, we claim that $\nu_{\epsilon} \to 0$ vaguely as $\epsilon \to 0$. Indeed, let $B$ be a Borel subset of $\R^{2}$ with $0 \notin \overline{B}$. Since $\nu_{\epsilon}$ is supported on the disk $B(0, \epsilon / \sigma(\epsilon))$ and $\lim_{\epsilon \to 0}\tfrac{\epsilon}{\sigma(\epsilon)}= 0$, there exists $\epsilon_{0} \in (0,1)$ such that $B \subseteq {\R^2 \setminus B(0,\epsilon / \sigma(\epsilon))}$ and $\nu_{\epsilon}(B)= 0$, for every  $\epsilon \in (0,\epsilon_{0})$.  Hence, $\lim_{\epsilon \to 0}\nu_{\epsilon}(B)= 0$ and this proves the vague convergence. Thus, we also have that 
\begin{align*}
\lim_{\epsilon \to 0}\int_{ x^{2}+y^{2} \geq h} x d\nu_{\epsilon}(x,y)= \lim_{\epsilon \to 0}\int_{ x^{2}+y^{2} \geq h} y d \nu_{\epsilon}(x,y)= 0
\end{align*}

Now, for $ \epsilon \in (0,1)$, if $m_{\epsilon}$ is the law of $Y_{\epsilon}(1)$, then it holds that $m_{\epsilon}= \id(0,0,\nu_{\epsilon})$ and following the notation of Theorem~\ref{thm:Kalenberg},
\begin{align*}
\alpha_{\epsilon}^{h}= \int_{|x|\leq h}xx^{T} d \nu_{\epsilon}(x) \quad\text{and}\quad \beta_{\epsilon}^{h}= -\int_{ |x| \geq h}x d\nu_{\epsilon}(x).
\end{align*}
Then, 
\begin{align*}
\lim_{\epsilon \to 0} \alpha_{\epsilon}^{h}= \alpha^{h} \quad\text{and}\quad \lim_{\epsilon \to 0} \beta_{\epsilon}^{h}= 0 \quad\text{where}\quad \alpha = \begin{pmatrix} 1/2 & r \\ r & 1/2 \end{pmatrix}.
\end{align*}
Therefore, Theorem~\ref{thm:Kalenberg} implies that $m_{\epsilon}$ converges weakly to  $m= \id(\alpha,0,0)$ as $\epsilon \to 0$. Consider the process 
\begin{align*}
R_{t}= (u B_{t}^{1}+ v B_{t}^{2}, v B_{t}^{1}+ u B_{t}^{2})
\end{align*}
where $(B^1,B^2)$ is a standard two-dimensional Brownian motion and 
\begin{align*}
 u = \frac{1}{2}\left(\sqrt{1/2+r}+\sqrt{1/2-r}\right),\quad v= \frac{1}{2}\left(\sqrt{1/2+r}-\sqrt{1/2-r}\right).
\end{align*}
Then $m$ is the law of $R_{1}$ and thus we have that $Y_{\epsilon}(1)$ converges weakly to $R_{1}$ as $\epsilon \to 0$.  Since $Y_{\epsilon}$ is a sequence of processes with independent stationary increments, we obtain that the laws of $(Y_{\epsilon}(t))_{t \in [0,1]}$ converge weakly to the law of $(R_t)_{t \in [0,1]}$ with respect to the Skorokhod topology as $\epsilon \to 0$.

Now, we are ready to prove Proposition~\ref{proposition:lower_bound}.

\begin{proof}[Proof of Proposition~\ref{proposition:lower_bound}]
We know that the pairs $(s,u)$ consisting of a time $s \geq 0$ and a jump size $u$ of $\wt{W}$ is a Poisson point process with intensity measure $\nu \otimes dt$ where $dt$ denotes Lebesgue measure on $\R_+$.  Thus for $ \epsilon , t>0$ it holds for every Borel subset $B$ of $\R^{2} \setminus \{0\}$ that 
\begin{align*}
|\{0\leq s \leq t :\Delta \wt{W}_{s}\in B\}|
\end{align*}
is a Poisson random variable with mean $t\nu(B)$.  Hence letting $B= \R^{2} \setminus \overline{B(0,\epsilon)}$ we have that
\begin{align}
\p[\sup_{0 \leq s \leq t}|\Delta \wt{W}_{s}|\leq \epsilon]
&= \p[|\{0\leq s \leq t: \Delta \wt{W}_{s} \in B\}|=0] \nonumber\\
&= \exp(-t\nu(B))= \exp(-t c_\alpha \epsilon^{-\alpha}) \quad\text{for every}\quad t,\epsilon>0  \label{eq:12}
\end{align}
and $c_\alpha > 0$ is a constant depending only on $\alpha$.  Fix $c > 0$.  Since $R_t$ is a (correlated) two-dimensional Brownian motion there exists $p_{0}>0$ such that 
\begin{align*}
\p[R_t \in [-c,c]^2 \ \text{for all}\ t \in [0,1]]= p_{0}.
\end{align*}
By the weak convergence of $Y_{\epsilon}$ to $R$ as $\epsilon \to 0$, we obtain that there exists $\epsilon_{0}>0$ such that 
\begin{align}\label{eq:13}
 \p[\wt{W}_{\epsilon}(t)\in [-c\sigma(\epsilon)/2, c\sigma(\epsilon)/2]^2 \ \text{for all} \   t \in [0,1]]\geq p_0 / 2 \quad \text{for all} \quad \epsilon \in (0,\epsilon_{0}]
\end{align}

Next,  we focus on the process $\wt{W}_1^{\epsilon}$.  Note that $W_t = \wt{W}_{\epsilon}(t) + N^{\epsilon}(t) - f_{\epsilon}(t)$ for all times $t$ and all $\epsilon \in (0,\epsilon_0]$,  where $f_{\epsilon}(t) = t\int_{|x| \geq \epsilon}xd\nu(x)$.  Note also that $f_{\epsilon}(t) =t \epsilon^{1-\alpha}I$ where $I = \int_{|x| \geq 1} x d\nu(x)$.  We pick $n = n(\epsilon) \in \N$ such that 
\begin{align*}
|f_{\epsilon}(t) - f_{\epsilon}(s)| = |t-s| \epsilon^{1-\alpha}| I| < \frac{a(\epsilon)}{2}
\end{align*}
for all $t,s \in [0,1]$ such that $|t-s| < 1/n$,  where $a(\epsilon) = c \sigma(\epsilon) / 2$.  More precisely,  we set $n = n(\epsilon) = \left \lfloor \frac{4|I|\epsilon^{-\alpha/2}}{c\sigma(1)} \right \rfloor + 1$. Then we have that
\begin{align*}
\p\!\left [ \sup_{0 \leq t \leq 1} |N^{\epsilon}(t) - f_{\epsilon}(t)| < a(\epsilon) \right ] \geq \p\!\left[ \sup_{t_j \leq t \leq t_{j+1}} |N^{\epsilon}(t) - f_{\epsilon}(t_j) | <\frac{a(\epsilon)}{2} \quad \text{for all} \quad  0 \leq j \leq n-1 \right] ,
\end{align*}
where $t_j = j / n$ for $0 \leq j \leq n-1$.  For each $0 \leq j \leq n-1$ we let $B_j$ be the event that $N^\epsilon|_{[t_j,t_{j+1}]}$ makes exactly one jump and $|N^\epsilon(t_{j+1}) - f_\epsilon(t_{j+1})| < a(\epsilon)/8$.  Let also $A_j$ be the event that $N^\epsilon|_{[0,1/n]}$ makes exactly one jump and $|N^\epsilon(1/n) - f_{\epsilon}(t_{j+1})| < a(\epsilon)/8$.  Then, since $N^{\epsilon}$ is a Markov process with stationary and independent increments, we obtain for each $1 \leq m \leq n-1$ that 
\begin{align}
\label{eqn:iteration}
\p\!\left[ \cap_{j=0}^{m} B_j \right ] = \E\!\left[ \left(\prod_{j=0}^{m-1}\one_{B_j} \right ) \p^{N^{\epsilon}(t_{m})}[A_m]  \right].
\end{align}
We claim that there exists $q > 0$ such that 
\begin{align*}
\p^{x}[ A_m] \geq q \quad\text{for all}\quad x \in B(f_{\epsilon}(t_{m}), a(\epsilon) / 8).
\end{align*}
Indeed, this follows because $|f_{\epsilon}(t_{m+1}) - f_{\epsilon}(t_m)|\geq a(\epsilon) / 3$ and $\frac{a(\epsilon)}{3}-\frac{a(\epsilon)}{4} =\frac{a(\epsilon)}{12}$ is larger than $\epsilon$ for all $\epsilon$ sufficiently small  and $\nu(B(f_\epsilon(t_{m+1})-x,a(\epsilon)/8)) > 0$ for all $x \in B(f_\epsilon(t_{m}),a(\epsilon)/8)$.  By decreasing the value of $\epsilon_0 > 0$ if necessary we can assume that this holds for $\epsilon_0$.  Iterating in \eqref{eqn:iteration} gives that
\begin{align}\label{eqn:positive_prob}
p(\epsilon) = \p[ \sup_{0 \leq t \leq 1} |N^{\epsilon}(t) - f_{\epsilon}(t)| < a(\epsilon) ] >0.
\end{align}
Let
\begin{align*}
A &= \{\wt{W}_{\epsilon_0}(t) \in [- c\sigma(\epsilon_0)/2, c\sigma(\epsilon_0)/2]^2\  \text{for all} \  t \in [0,1] \} \quad\text{and}\\
B &=\{\sup_{0 \leq t \leq 1} |N^{\epsilon_0}(t) - f_{\epsilon_0}(t)| < a(\epsilon_0) \}.
\end{align*}
On $A \cap B$ we have that $W_t \in [-c \sigma(\epsilon_0),c \sigma(\epsilon_0)]^2 $ for all $t \in [0,1]$.

By the independence of $\wt{W}_{\epsilon_{0}}$ and $\wt{W}-\wt{W}_{\epsilon_{0}}$, and \eqref{eq:13}, \eqref{eqn:positive_prob}, we obtain that 
\begin{align}\label{eq:14}
 \p_{(0,0)}[W_t \in [-c\sigma(\epsilon_0),c \sigma(\epsilon_0)]^2 \ \text{for all}\  t \in [0,1]] \geq \p[A \cap B] \geq p(\epsilon_0) p_0 / 2.
\end{align}
Finally, by the convergence of  the process $a_n^{-1} (X_{nt},Y_{nt})$ for $t\in [0,1]$ to $(W_{t})_{t \in [0,1]}$ (Theorem~\ref{thm:boundary_length_conv}) and~\eqref{eq:14}, we obtain that there exist $p>0$ and $n_{0}$ such that 
\begin{align*}
 \p_{(0,0)}\left[a_n^{-1} (X_k,Y_k) \in [-c \sigma(\epsilon_0),c \sigma(\epsilon_0)]^2 \ \text{for all} \   1\leq k \leq n \right] \geq p \quad\text{for all} \quad n \geq n_0.
\end{align*} 
By decreasing the value of $\epsilon_0 > 0$ if necessary we may assume that $\sigma(\epsilon_0) < 1$ so that the probability above is only increased if replace $c \sigma(\epsilon_0)$ by $c$.  Moreover, by possibly decreasing the value of $p > 0$ we can take $n_{0}=1$ and altogether this completes the proof.
\end{proof}

Next, we give the proofs of Proposition~\ref{proposition:lower_bound_conditioning} and Theorem~\ref{thm:local_convergence}.

\begin{proof}[Proof of Proposition~\ref{proposition:lower_bound_conditioning}]
Fix $A,B > 0$ and let $A^1,A^2 > 0$ be so that $A = A^1 + A^2$.  Take $c > 0$ as in Proposition~\ref{proposition:lower_bound} and assume that $c < \min(A^1,A^2)$.  Let $A_n$, $A_n^1$, $A_n^2$, $B_n$ be as defined before and let $C_n = \lfloor C n^{1/\alpha} \rfloor$ where $C > A^2 + B + c$. Consider also the following events: 
\begin{align*}
&F_n = \{(X_k - X_{k-1}, Y_k - Y_{k-1}) = (1,-1) \  \text{for all} \  1 \leq k \leq A_n^1\}, \\
&G_n = \{(X_k , Y_k) \in R_n \  \text{for all} \  A_n^1 \leq k \leq n - C_n\},\\
&K_n = \{X_k , Y_k \geq 0 \  \text{for all} \  n - C_n \leq k \leq n, \, X_n = B_n,\, Y_n = 0\},
\end{align*}
where 
\begin{align*}
R_n = [-c a_n + A_n^1, c a_n + A_n^1] \times [-c a_n + A_n^2, c a_n + A_n^2].
\end{align*}
Note that since $c < \min(A^1,A^2)$ we have that $R_n \subseteq \N_0^2$ for all $n$ large enough.  For $n$ sufficiently large, we have that 
\begin{align}\label{eq:15}
\p_{(0,A_n)}[ X_k , Y_k \geq 0 \  \text{for all} \  1 \leq k \leq n,\, X_n = B_n, \, Y_n = 0] \geq \p_{(0, A_n)}[ F_n \cap G_n \cap K_n ]
\end{align}
and the right hand side of~\eqref{eq:15} is bounded from below by 
\begin{align}\label{eq:16}
\min_{(x,y) \in R_n}\p[ K_n \,| \, (X_{n- C_n}, Y_{n - C_n}) =(x,y)] \p_{(A_n^1, A_n^2)}[ G_n ] p_{0}^{ A_n^1}
\end{align}
The middle term of the product in \eqref{eq:16} is bounded below uniformly in $n$ by a constant $p > 0$ by Proposition~\ref{proposition:lower_bound}.  We turn to give a lower bound for the first term on the left hand side of~\eqref{eq:16}.  Suppose that the walk starts from $(x,y) \in R_n$.
\begin{itemize}
\item If it has an increment of type $(0,B_n)$ (i.e., a $\move_{0,B_n}$ move) it ends up at $(x,y+B_n)$.
\item If it then has $y+B_n$ increments of type $(1,-1)$ (i.e., $\move_e$ moves) it ends up at $(x+y+B_n,0)$.
\item If it then has an increment of type $(-x-y,0)$ (i.e., a $\move_{x+y,0}$ move) it ends up at $(B_n,0)$.
\item If it then has $C_n - (y+B_n) - 2$ increments of type $(0,0)$ (i.e., $\move_{0,0}$ moves) it altogether ends up at $(B_n,0)$ after $C_n$ steps.  Note that since $C > A^2 + B + c$ and $y \leq (c + A^2) n^{1/\alpha} + 1$ as $(x,y) \in R_n$ we have that $C_n - (y+B_n) - 2 \geq 0$ for all $n$ large enough.
\end{itemize}
We thus see that if $(x,y) \in R_n$ then we have that
\begin{align}\label{eq:17}
\p[K_n \, | \, X_{n - C_n} = x,\, Y_{n- C_n} = y ] \geq p_{B_n+2} \cdot p_0^{y+B_n} \cdot p_{x+y+2} \cdot p_2^{C_n-(y+B_n+2)}.
\end{align}
Note that $p_{B_n+2}$ and $p_{x+y+2}$ for $(x,y) \in R_n$ are both at least a constant times $n^{-1-2/\alpha}$.  Altogether, we see from~\eqref{eq:17} that there exists $q \in (0,1)$ such that 
\begin{align}
\label{eq:18}
\min_{x,y \in R_n}\p[ K_n \,| \, (X_{n- C_n}, Y_{n - C_n}) =(x,y)] \geq q^{a_n}.
\end{align}
Then the proof is complete by combining~\eqref{eq:15}, \eqref{eq:16}, and~\eqref{eq:18} with Proposition~\ref{proposition:lower_bound}.
\end{proof}

\begin{proof}[Proof of Theorem~\ref{thm:local_convergence}]
Fix $b = (b_0,\ldots,b_{2r}) \in S^{2r+1}$ and for $0\leq i\leq 2r$, $m \in \N$, we define 
\begin{align*}
A_{m}^{i}= \{\omega \in S^{\N_{0}}:\omega(i+j)= b_{j-(m-1)(2r+1)} \  \text{for all} \  (m-1)(2r+1)\leq j \leq m(2r+1)-1\}
\end{align*}
and set
\begin{align*}
C_{k}^{i}= \sum_{m=1}^{k}\one_{A_{m}^{i}}.
\end{align*}
Then for each $0\leq i \leq 2r$ we have that the sequence $(\one_{A_{m}^{i}})_{m \in \N}$ is i.i.d.\ with mean $p= \mu_{2r+1}(b)$. By Cramer's theorem \cite{dz2010ld}, there exists a constant $c= c(p,\epsilon) > 0$ depending only on $p$ and $\epsilon$ such that 
\begin{align}
\label{eq:21}
\p\!\left[ \left |\frac{C_{k}^{i}}{k}-p\right | \geq \epsilon \right] \leq e^{-c k} \quad  \text{for all} \quad  k \in \N.
\end{align}

Let
\begin{align*}
F_n = \{ X_k,\, Y_k \geq 0 \ \text{for all}\  1 \leq k \leq n,X_n = B_n, Y_n = 0\}.
\end{align*}
If $k_{n}\in \N$, $0 \leq \ell_{n}\leq 2r$, and $n - 2r= \ell_{n}+k_{n}(2r+1)$ then we have that 
\begin{align*}
m_{n,2r+1}(b)&= \frac{1}{n-2r+1}\sum_{i= 0}^{2r}\sum_{k= 1}^{k_{n}}\p_{(0,A_n)}[A_{k}^{i} \, | \,F_{n}]
 +\frac{1}{n-2r+1}\sum_{i= 0}^{\ell_{n}}\p_{(0,A_n)}[A_{k_{n}+1}^{i} \, |\,F_{n}].
\end{align*}
Hence, it holds that 
\begin{align}\label{eq:22}
&|m_{n,2r+1}(b)-\mu_{2r+1}(b)| \nonumber \\
\leq& \frac{k_{n}}{n-2r+1}\sum_{i= 0}^{2r}\E_{(0,A_n)}\left [\left | \frac{C_{k_{n}}^{i}}{k_{n}}-p \right |\, | \,F_n \right ] + \frac{1}{n-2r+1}\sum_{i= 0}^{\ell_{n}}|\p_{(0,A_n)}[A_{k_{n} + 1}^{i} \,| \, F_n ] - p | \nonumber \\
\leq& \frac{k_n}{n-2r+1}\sum_{i=0}^{2r}\E_{(0,A_n)} \left [ \left | \frac{C_{k_n}^{i}}{k_n} - p \right | \, | \, F_n \right ] + \frac{2(2r+1)}{n-2r+1}.
\end{align}
Also, Proposition~\ref{proposition:lower_bound_conditioning} implies that there exists a constant $q \in (0,1)$ such that 
\begin{align*}
\p_{(0,A_n)}[ F_n ] \geq q^{a_n} \quad\text{for all}\quad n \in \N
\end{align*}
and so \eqref{eq:21} implies that
\begin{align}\label{eq:23}
\E_{(0,A_n)}\left [ \left | \frac{C_{k_n}^i}{k_n} - p \right | \, | \, F_n \right] &\leq \epsilon + (1+p) \p_{(0,A_n)}\left [ \left | \frac{C_{k_n}^i}{k_n} - p \right | \geq \epsilon \, | \, F_n \right] \nonumber \\
&\leq \epsilon + (1 + p) e^{-c k_n} q^{-a_n}.
\end{align}
In the above, we bounded the expectation by considering the cases that $|C_{k_n}^i/k_n - p|$ is either smaller or larger than $\epsilon$ and also used the upper bound $|C_{k_n}^i/k_n - p| \leq 1+p$.  Therefore~\eqref{eq:22} and~\eqref{eq:23} imply that $\lim_{n \to \infty} m_{n,2r+1}(b) = \mu_{2r+1}(b)$ and so 
\begin{align}\label{eq:24}
\lim_{n \to \infty} m_{n,2r+1}(C) = \mu_{2r+1}(C) \quad\text{for all} \quad C \subseteq {S^{2r+1}} \quad\text{finite}.
\end{align}

Fix $\delta>0$. Then there exist $i_{0},j_{0}\in \N$ such that $ \mu_{2r+1}(B) \geq 1 - \delta$ where 
\begin{align*}
B= \{(1,-1),\,(-i,j),\,0\leq i \leq i_{0} , \,0\leq j \leq j_{0}\}
\end{align*}
Since $B$ is finite, \eqref{eq:24} implies for all $n$ sufficiently large that 
\begin{align*}
|m_{n,2r+1}(C)-\mu_{2r+1}(C)|\leq \delta \quad\text{for all}\quad C \subseteq B.
\end{align*} 
Moreover, we note for all $n$ sufficiently large that 
\begin{align*}
|m_{n,2r+1}(B^{c})-\mu_{2r+1}(B^{c})| = |m_{n,2r+1}(B)-\mu_{2r+1}(B)| \leq \delta. 
\end{align*}
Then for $C \subseteq S^{2r+1}$ and all $n$ sufficiently large we have that 
\begin{align*}
|m_{n,2r+1}(C)-\mu_{2r+1}(C)|&\leq |m_{n,2r+1}(C\cap{B})-\mu_{2r+1}(C\cap{B})|+|m_{n,2r+1}(C\cap{B^{c}})-\mu_{2r+1}(C\cap{B^{c}})|\\ 
&\leq \delta +m_{n,2r+1}(B^{c})+\mu_{2r+1}(B^{c})\\
&\leq \delta +|m_{n,2r+1}(B^{c})-\mu_{2r+1}(B^{c})|+2\mu_{2r+1}(B^{c})\\
& \leq 4\delta. 
\end{align*}
This completes the proof as $\delta>0$ was arbitrary.
\end{proof}

\section{Benjamini-Schramm convergence}
\label{sec:B-S_convergence}

We suppose that we have the same setup as in the beginning of Section~\ref{sec:total_variation}.  That is, we fix $\alpha \in (1,2)$ and assume that we have chosen weights $\weight_k$ satisfying the statement of Lemma~\ref{lem:condition_on_a_k}.  Fix $A, B > 0$, let $A_n = \lfloor A n^{1/\alpha} \rfloor$, and $B_n = \lfloor B n^{1/\alpha} \rfloor$.  Let $G_n$ be a $\BPRPM$ with $n+1$ edges sampled with probability proportional to its weight and conditioned such that it has $A_n + 1$ (resp.\ $B_n + 1$) edges on its western (resp.\ eastern) boundary. Let $w_{0}^{n},\ldots,w_{n-1}^{n} \in S$ be the steps which encode $G_n$. Let also $(U_n)$ be a sequence of independent random variables where $U_n$ is uniform in $\{0,\ldots,n-1\}$ and we assume that $(U_n)$ is independent of $(G_n)$. Let $\rho_n$ be the active vertex corresponding to the step~$w_{U_{n}}^{n}$ during the construction of $G_{n}$.  Then $(G_{n},\rho_{n})$ is a rooted random graph. In this section, we are going to prove Theorem~\ref{thm:benjamini_schramm_conv} which we recall asserts that $(G_{n},\rho_{n})$ converges in the Benjamini-Schramm sense (Definition~\ref{def:graph_convergence}) to an infinite volume rooted random planar map $(G,\rho)$ which is encoded by a bi-infinite sequence $(w_k)_{k \in \Z}$ of i.i.d.\ steps in $S$ chosen using the probabilities $p_{0}$, $p_{i+j+2}$ for $i,j \in \N_{0}$.

\subsection{Definition of the infinite volume $\BPRPM$}

Let us now explain how $(G,\rho)$ is constructed from $(w_k)_{k \in \Z}$. Suppose that we have an infinite boundary in the plane consisting of vertices and edges along with a root vertex such that all the vertices to the left of the root vertex are not part of any edge while all the adjacent vertices on the boundary and to the right of the root vertex are connected by an edge. Concretely, we imagine that the boundary is the real line and the vertices are the integers where we take $0$ to be the root vertex. Then for $i\in \N_{0}$ there exists exactly one (directed) edge emanating from $i$ connecting it with $i+1$ and for $i \in \Z$ with $i < 0$ there are no edges emanating from $i$.  We will construct a planar map drawn in the upper half-plane with the above boundary encoded by the sequence of moves $(w_k)_{k \in \N_0}$.  Similarly, we will construct a planar map drawn in the lower half-plane with the above boundary encoded by the sequence of moves $(w_{-k})_{k \in \N}$.

We will first describe the construction of the part of the planar map drawn in the upper half-plane. We start with $0$ as the active vertex.  An $\move_{e}$ move will sew an edge to the current boundary from and to the left of the active vertex and then move the active vertex to the upper endpoint of the added edge (i.e., to the left).  An $\move_{i,j}$ move will sew an open face with $i+1$ edges on its west and $j+1$ edges on its east where the north of the face is sewn to the active vertex and the west of the face is sewn to the boundary consisting of the vertices to the right of the active vertex.  We then sew a new edge to the southernmost east edge of the added face.  Then the new active vertex is the upper vertex of this edge.  Repeating the above steps, we construct an infinite volume planar map drawn in the upper half-plane.

We will now describe the construction of the part of the planar map which is drawn in the lower half-plane.  The construction is given by performing the ``reverse'' of the procedure described in the previous paragraph.  As before, we take $0$ to be our initial active vertex.  An $\move_{e}$ move removes the edge from the current boundary of the planar map which connects the active vertex to the vertex which is immediately to its right and moves the active vertex to the lower endpoint of the removed edge.  An $\move_{i,j}$ move sews an open face with $i+1$ edges on its west and $j+1$ edges on its east, sewing the south of the face to the vertex exactly to the right of the active vertex and the east of the face to the part of the boundary consisting of the vertices lying to the left of the vertex which lies exactly to the right of the active vertex. The new active vertex will be the north of the added face.  To complete the construction of~$G$, we repeat the above steps to obtain a planar map which is drawn in the lower half-plane.

We note that the above construction is translation invariant, i.e.,  for any $m \in \Z$ the random planar maps encoded by the sequences of steps $(w_{k+m})_{k \in \Z}$ and $(w_k)_{k \in \Z}$ have the same distribution.

\subsection{Ball absorption time bounds}

In the forward construction, during an $\move_{i,j}$ step, all the vertices strictly between the lower endpoint of the edge we added and the old active vertex will be absorbed in the interior of the structure lying below the new boundary that has been formed.  In every step, we keep track of the distance along the boundary of each vertex in the boundary from the active vertex.  After an~$\move_{e}$ move, the distance of each vertex to the left (resp.\ right) of the old active vertex will decrease (resp.\ increase) by $1$.  After an $\move_{i,j}$ move, vertices to the right of the old active vertex and at distance between $1$ and $i$ will be absorbed into the interior of the structure.  The distance of a vertex to the right of and at distance at least $i+1$ from the old active vertex will  decrease by $i$.  The distance of a vertex to the left of the old active vertex will increase by $j$.

Similarly, in the reverse construction, we keep track of the distance along the boundary of each boundary vertex from the active vertex. After an $\move_{e}$ move, the distance of every vertex which is to the left (resp.\ right) of the old active vertex will increase (resp.\ decrease) by $1$.  After an $\move_{i,j}$ step, the distance of a vertex to the right of the old active vertex increases by $i$ while the distance of a vertex which is to the left of and with distance to the old active vertex at least $j$ will decrease by~$j$.  Vertices which are to the left of the old active vertex and with distance at most $j-1$ will be absorbed into the interior of the structure.

Clearly, there is a bijection between the construction of the planar map encoded by $(w_k)_{k \in \Z}$ and a lattice walk $(X_k,Y_k)_{k \in \Z}$ normalized so that $X_0 = 0$, $Y_0 = 0$, where for each $k \in \Z$ the increment $(X_k - X_{k-1},Y_k-Y_{k-1})$ is $(1,-1)$ if $w_k = \move_e$ and is $(-i,j)$ if $w_k = \move_{i,j}$.

Next, we examine the time interval during which a vertex stays on the boundary before it gets absorbed during the forward construction.  Suppose that $v$ is a vertex on the boundary.  Let $T_1$ be the first time that $v$ is either the active vertex or $v$ is absorbed into the interior.  Assume that we have defined $T_1,\ldots,T_m$.  We let $T_{m+1}$ be equal to the first time after $T_m$ that $v$ is either absorbed in the interior or is equal to the active vertex.  We take the convention that if $v$ is in the interior at the time $T_m$ then $T_{m+1} = T_m$.  We claim that $\p[ T_m < \infty] = 1$ for all $m \in \N$ and that if $N$ is the smallest $m$ so that $v$ is absorbed in the interior at the time $T_m$ then $\p[N < \infty]$.  Once we have proved these claims, we will have shown that $v$ is a.s.\ absorbed into the interior.

Let us first prove that $\p[T_1 < \infty] = 1$.  Assume that $v$ is to the left of and has distance $j$ to the active vertex.  Then $T_1 = \infty$ implies that $Y_k \geq -j$ for all $k \in \N$.  Fix $\epsilon > 0$.  Then for $n \in \N$ large enough we have that
\begin{equation}
\label{eqn:t1_finite}
\p[ \min_{1 \leq k \leq n} Y_k \geq -j] \leq  \p[ a_n^{-1} \inf_{0 \leq t \leq 1} Y_{t n} \geq -\epsilon ].
\end{equation}
As $n \to \infty$, we have that $(a_n^{-1} Y_{t n})_{t \in [0,1]}$ converges to $(W_t^2)_{t \in [0,1]}$.  Therefore the right hand side of~\eqref{eqn:t1_finite} converges to $\p[ \inf_{0 \leq t \leq 1} W_t^2 \geq -\epsilon]$ as $n \to \infty$.  Since $\p[ \inf_{0 \leq t \leq 1} W_t^2 \geq -\epsilon]$ tends to $0$ as $\epsilon \to 0$ \cite[Chapter~VIII, Proposition~2]{bertoin1996levy}, we conclude that $\p[ \min_{1 \leq k \leq n} Y_k \geq -j] \to 0$ as $n \to \infty$.  Therefore $\p[T_1 < \infty] = 1$.

In the next time step after $T_1$, if $v$ has not been absorbed into the interior then $v$ will either be to the left or to the right of the active vertex.  If it is to the left of the active vertex, then the same argument given above implies that $T_2 < \infty$ a.s.  If it is to the right of the active vertex, then the event that $T_2 < \infty$ corresponds to the event that $X_{k+T_1} - X_{T_1}$ for $k \in \N_0$ goes below $-1$.  Since $X_{k+T_1} - X_{T_1}$ has the same distribution as $X_k$, it suffices to work with the process $X_k$.  Fix $\epsilon > 0$.  Then for $n \in \N$ large enough we have that
\begin{equation}
\label{eqn:t2_finite}
\p[ \min_{1 \leq k \leq n} X_k \geq -1] \leq  \p[ a_n^{-1} \inf_{0 \leq t \leq 1} X_{t n} \geq -\epsilon].
\end{equation}
As $n \to \infty$, we have that $(a_n^{-1} X_{t n})_{t \in [0,1]}$ converges to $(W_t^1)_{t \in [0,1]}$.  Therefore the right hand side of~\eqref{eqn:t2_finite} converges to $\p[ \inf_{0 \leq t \leq 1} W_t^1 \geq -\epsilon]$ as $n \to \infty$.  Since $\p[ \inf_{0 \leq t \leq 1} W_t^1 \geq -\epsilon]$ tends to $0$ as $\epsilon \to 0$ \cite[Chapter~VIII, Proposition~2]{bertoin1996levy}, we conclude that $\p[ \min_{1 \leq k \leq n} X_k \geq -1] \to 0$ as $n \to \infty$.  Altogether, we have shown that $\p[T_2 < \infty] = 1$.  Applying this inductively, we see that $\p[ T_m < \infty] = 1$ for all $m \in \N$.

It is left to explain why $\p[N < \infty] = 1$.  This, however, is easy to see because each time $v$ is the active vertex there is a positive chance that the next two moves are $\move_e$, $\move_{3,1}$ in which case $v$ is absorbed into the interior.  That is, $N$ is stochastically dominated by a geometric random variable.  This completes the proof of the two claims.

\newcommand{\off}{\mathrm{off}}

Let $T_{v,\off} = T_N$ where $(T_m)$, $N$ are as defined above.  Then we have shown that $\p[ T_{v,\off} < \infty] = 1$.   We also let $\wh{T}_{v,\off}$ be the first time that $v$ is absorbed into the interior of the structure when performing the reverse procedure.  Then arguing as above we also have that $\p[\wh{T}_{v,\off} < \infty] = 1$.

Now, let $\rho$ be the active vertex corresponding to the move $w_{0}$ of the infinite volume random planar map. Then $(G,\rho)$ is a random rooted graph.  In what follows, for integers $m \leq n$ we will sometimes use the notation $(w_k)_{k=m}^n$ to denote both the sequence of moves $(w_m,\ldots,w_n)$ and the part of $G$ which is encoded by the sequence $(w_m,\ldots,w_n)$.  We will use the notation $(w_k)_{k \geq m}$ and $(w_k)_{k \leq m}$ in the same way but with respect to the sequences of moves $(w_m,\ldots)$ and $(\ldots,w_m)$, respectively.

\begin{proposition}
\label{prop:local_finite}
For each $\epsilon>0, r\in \N$ there exists $m \in \N$ such that 
\begin{align*}
\p[B_{G}(\rho,r)\subseteq (w_k)_{k=-m}^m,\,A_{m,r}] \geq 1- \epsilon
\end{align*}
 where $B_{G}(\rho,r)$ is the ball of $G$ of radius $r$ centered at $\rho$ with respect to the graph distance and $A_{m,r}$ is the event that every vertex in $B_{G}(\rho,r)$ has been absorbed in the interior of $(w_k)_{k=-m}^m$.
\end{proposition}
\begin{proof}
We will prove the claim using induction on $r$. We first consider the case $r = 1$.  Fix $\epsilon \in (0,1)$. Note that for $m \in \N$, if the event $C_m = \{T_{\rho,\text{off}} \leq m,\ \wh{T}_{\rho,\text{off}} \leq m\}$ occurs (recall the notation used to analyze the boundary behavior of $G$), the vertex $\rho$ gets absorbed in the interior of the planar map encoded by the sequence of steps $(w_k)_{k=-m}^m$. Since in the construction of the map neighbors are added to~$\rho$ only as long as it stays on the boundary, if $C_m$ occurs, we have that
\begin{align*}
B_G(\rho,1) \subseteq (w_k)_{k=-m}^m.
\end{align*}
We therefore have that $\p[ B_G(\rho,1) \subseteq (w_k)_{k=-m}^m] \to 1$ as $m \to \infty$.  Take $m$ sufficiently large so that $\p[ B_G(\rho,1) \subseteq (w_k)_{k=-m}^m] \geq 1-\epsilon$.  Then since each of the steps $(w_k)_{k=-m}^m$ involves adding at most finitely many vertices, the set of vertices on the boundary of the map encoded by $(w_k)_{k \geq -m}$ which are in $B_G(\rho,1)$ is finite.  Likewise, the set of vertices on the boundary of the map encoded by $(w_k)_{k \leq m}$ which are in $B_G(\rho,1)$ is also finite.  By the translation invariance of the law of the map encoded by $(w_k)_{k \in \Z}$, the amount of additional time after a vertex is first discovered by the exploration that it gets absorbed into the interior is a.s.\ finite.  Altogether, this implies that we can increase the value of $m$ if necessary so that $\p[ B_G(\rho,1) \subseteq (w_k)_{k=-m}^m, A_{m,1}] \geq 1-\epsilon$.

Suppose that $R \in \N$ and the result holds for all $1 \leq r \leq R$ and $\epsilon > 0$.  This implies that $\p[ B_G(\rho,R) \subseteq (w_k)_{k=-m}^m, A_{m,R}] \to 1$ as $m \to \infty$.  Suppose $m \in \N$ is such that $B_G(\rho,R) \subseteq (w_k)_{k=-m}^m$ and $A_{m,R}$ occur.  Then since each of the steps $(w_k)_{k=-m}^m$ involves adding at most finitely many vertices, the set of vertices on the boundary of the map encoded by $(w_k)_{k \geq -m}$ which are in $B_G(\rho,R+1)$ is finite.  Likewise, the set of vertices on the boundary of the map encoded by $(w_k)_{k \leq m}$ which are in $B_G(\rho,R+1)$ is also finite.  By the translation invariance of the law of the map encoded by $(w_k)_{k \in \Z}$, the amount of additional time after a vertex is first discovered by the exploration that it gets absorbed into the interior is a.s.\ finite.  Altogether, this implies that we can increase the value of $m$ if necessary so that $\p[ B_G(\rho,R+1) \subseteq (w_k)_{k=-m}^m, A_{m,R+1}] \geq 1-\epsilon$.
\end{proof}

\begin{remark}
An immediate consequence of Proposition~\ref{prop:local_finite} is that  $B_{G}(\rho,r)$ is finite for all $r\in \N$ a.s.  Moreover, by the translation invariance of the infinite volume planar map and the above observation, we obtain that $B_{G}(v,r)$ is finite for all vertices $v$ and $r \in \N$.
\end{remark}

\subsection{Conclusion of the proof}
We are now ready to prove the Benjamini-Schramm convergence of the sequence of random rooted graphs $(G_n,\rho_n)$ to $(G,\rho)$ as $n \to \infty$.

\begin{proof}[Proof of Theorem~\ref{thm:benjamini_schramm_conv}]
Fix $r,\epsilon>0$ and a connected rooted graph $(H,\rho')$.  Let $A_{m,r}$ be as in Proposition~\ref{prop:local_finite} and let
\begin{align*}
B_{m,r}= A_{m,r} \cap \{B_{G}(\rho,r)\subseteq (w_k)_{k=-m}^m \} \quad\text{and}\quad C_{m,r} = B_{m,r} \cap \{B_{G}(\rho,r)=(H,\rho')\}.
\end{align*}
By Proposition~\ref{prop:local_finite}, there exists $m \in \N$ such that $\p[B_{m,r}] \geq 1 - \epsilon$.  We also observe that for all $n \in \N$, if we condition on the event $U_{n} \in [m,n-m]$, we have that if $(w_k^n)_{k=U_n-m}^{U_n+m} \in C_{m,r}$ then $B_{G_{n}}(\rho_{n},r)= (H,\rho')$ and if $(w_k^n)_{k=U_n-m}^{U_n+m} \in B_{m,r}$ then 
\begin{align*}
B_{G_{n}}(\rho_{n},r)\subseteq (w_k^n)_{k=U_n-m}^{U_n+m} 
\end{align*}
and every vertex of $B_{G_{n}}(\rho_{n},r)$ has been absorbed in the interior of $G_{n}$ during these steps.  By combining the above observations and for $n$ sufficiently large, we have that 
\begin{align}
\label{eq:36}
&\p[B_{G_n}(\rho_n, r) = (H,\rho')] \nonumber\\
\leq&\frac{n-2m}{n}\p[\{B_{G_n}(\rho_n,r) = (H,\rho')\} \cap \{ (w_k^n)_{k=U_n-m}^{U_n+m} \in B_{m,r}^c \}\,|\,U_{n} \in [m,n-m]] + \nonumber \\
&\frac{n-2m}{n}\p[\{B_{G_{n}}(\rho_n,r) = (H,\rho')\} \cap \{ (w_k^n)_{k=U_n-m}^{U_n+m} \in B_{m,r}\} \,|\, U_{n} \in [m,n-m]] + \frac{2m}{n} \nonumber \\
\leq& \nu_{n}^m(B_{m,r}^c) + \nu_{n}^m (C_{m,r}) + \frac{2m}{n}
\end{align}
where $\nu_{n}^m$ is the law of $(G_n,\rho_n)$ conditioned on the event $\{U_n \in [m,n-m] \}$. As $n \to \infty$, Theorem~\ref{thm:local_convergence} implies that the right hand side of \eqref{eq:36} converges to $\mu_{2m+1}(B_{m,r}^{c})+\mu_{2m+1}(C_{m,r})$ (recall the definition of the measures $\mu_{m}$ from Section~\ref{sec:total_variation}) and hence 
\begin{align}\label{eq:37}
\limsup_{n\to \infty}\p[B_{G_{n}}(\rho_{n},r)= (H,\rho')]&\leq \mu_{2m+1}(B_{m,r}^{c})+\mu_{2m+1}(C_{m,r})
 \leq \epsilon +\mu_{2m+1}(C_{m,r})
\end{align}
 and
\begin{align}
\label{eq:38}
\mu_{2m+1}(C_{m,r})\leq \liminf\limits_{n\to \infty}\p[B_{G_{n}}(\rho_{n},r)= (H,\rho')]
\end{align}
since $\mu_{2m+1}(B_{m,r}) = \p[B_{m,r}]$. Moreover
\begin{align}\label{eq:39}
\mu_{2m+1}(C_{m,r})\leq \p[B_{G}(\rho,r)= (H,\rho')]&\leq \mu_{2m+1}(C_{m,r})+\mu_{2m+1}(B_{m,r}^{c}) 
 \leq \epsilon +\mu_{2m+1}(C_{m,r})
\end{align}
Since $\epsilon > 0$ was arbitrary, \eqref{eq:37}, \eqref{eq:38} and \eqref{eq:39} imply that 
\begin{align*}
\p[B_{G}(\rho,r)= (H,\rho')]= \lim_{n\to \infty}\p[B_{G_{n}}(\rho_{n},r)= (H,\rho')]
\end{align*}
This completes the proof.
\end{proof}

\bibliographystyle{abbrv}
\bibliography{biblio}

\end{document}